\documentclass{amsart}
\usepackage{graphicx, color}
\usepackage{amscd}
\usepackage{amsmath,empheq}
\usepackage{amsfonts}
\usepackage{amssymb}
\usepackage{mathrsfs}

\usepackage[all]{xy}

\newtheorem{theorem}{Theorem}

\newtheorem{lemma}[theorem]{Lemma}
\newtheorem{prop}[theorem]{Proposition}
\newtheorem{remark}{Remark}

\newtheorem{claim}{Claim}

\newtheorem{definition}[theorem]{Definition}

\newenvironment{proof-sketch}{\noindent{\bf Sketch of Proof}\hspace*{1em}}{\qed\bigskip}

\everymath{\displaystyle}

\newcommand{\RR}{\mathbb R}
\newcommand{\NN}{\mathbb N}

\newcommand{\ZZ}{\mathbb Z}

\renewcommand{\leq}{\leqslant}

\renewcommand{\geq}{\geqslant}

\begin{document}

\title[Resonant Robin problems driven by the $p$-Laplacian plus]{Resonant Robin problems driven by the $p$-Laplacian plus an indefinite potential}

\author[N.S. Papageorgiou]{Nikolaos S. Papageorgiou}
\address[N.S. Papageorgiou]{National Technical University, Department of Mathematics,
				Zografou Campus, Athens 15780, Greece}
\email{\tt npapg@math.ntua.gr}

\author[V.D. R\u{a}dulescu]{Vicen\c{t}iu D. R\u{a}dulescu}
\address[V.D. R\u{a}dulescu]{Department of Mathematics, Faculty of Sciences, King Abdulaziz University,
P.O. Box 80203, Jeddah 21589, Saudi Arabia \&  Department of Mathematics, University of Craiova, Street A.I. Cuza 13,
          200585 Craiova, Romania}
\email{\tt vicentiu.radulescu@imar.ro}

\author[D.D. Repov\v{s}]{Du\v{s}an D. Repov\v{s}}
\address[D.D. Repov\v{s}]{Faculty of Education and Faculty of Mathematics and Physics,
University of Ljubljana, SI-1000 Ljubljana, Slovenia}\email{dusan.repovs@guest.arnes.si}


\begin{abstract}
We consider a nonlinear Robin problems driven by the $p$-Laplacian plus an indefinite potential. The reaction is resonant with respect to a variational eigenvalue. For the principal eigenvalue we assume strong resonance. Using variational tools and critical groups we prove existence and multiplicity theorems.
\end{abstract}

\maketitle


\section{Introduction}

Let $\Omega\subseteq\RR^N$ be a bounded domain with a $C^2$-boundary $\partial\Omega$. In this paper, we study the following nonlinear Robin problem
\begin{equation}\label{eq1}
	\left\{\begin{array}{ll}
		-\Delta_p u(z)+\xi(z)|u(z)|^{p-2}u(z)=f(z,u(z))&\mbox{in}\ \Omega,\\
		\frac{\partial u}{\partial n_p}+\beta(z)|u|^{p-2}u=0&\mbox{on}\ \partial\Omega\,.
	\end{array}\right\}
\end{equation}

In this problem, $\Delta_p$ denotes the $p$-Laplace differential operator defined by
$$\Delta_pu={\rm div}\,(|Du|^{p-2}Du)\ \mbox{for all}\ u\in W^{1,p}(\Omega).$$

The potential function $\xi(\cdot)\in L^{\infty}(\Omega)$ is indefinite (that is, sign changing) and the reaction term $f(z,x)$ is a Carath\'eodory function (that is, for all $x\in\RR$, $z\mapsto f(z,x)$ is measurable and for almost all $z\in\Omega$, $x\mapsto f(z,x)$ is continuous). In the boundary condition, $\frac{\partial u}{\partial n_p}$ denotes the generalized normal derivative corresponding to the $p$-Laplace differential operator and is defined by extension of the map
$$u\mapsto\frac{\partial u}{\partial n_p}=|Du|^{p-2}(Du, n)_{\RR^N}\ \mbox{for all}\ u\in C^1(\overline{\Omega}),$$
with $n(\cdot)$ being the outward unit normal on $\partial\Omega$. The boundary coefficient term is $\beta\in C^{0,\alpha}(\partial\Omega)$ with $\alpha\in(0,1)$ and $\beta(z)\geq 0$ for all $z\in\partial\Omega$. The case $\beta\equiv 0$ corresponds to the Neumann problem.

Our aim here is to investigate the existence and multiplicity of nontrivial smooth solutions for problem (\ref{eq1}) when resonance occurs, namely when the function $\frac{f(z,x)}{|x|^{p-2}x}$ asymptotically as $x\rightarrow\pm\infty$ hits a variational eigenvalue of $-\Delta_p+\beta(z)I$ with Robin boundary condition (here $I$ denotes the identity operator). In the case of resonance with respect to the principal (first) eigenvalue, we consider problems with ``strong" resonance, namely we have
$$f(z,x)=\hat{\lambda}_1|x|^{p-2}x+g(z,x),$$
with $\hat{\lambda}_1$ being the first eigenvalue and $g(z,x)$ is a Carath\'eodory perturbation satisfying
$$\lim\limits_{x\rightarrow\pm\infty}g(z,x)=0\ \mbox{and}\ \lim_{x\rightarrow\pm\infty}\int^x_0g(z,s)ds\in\RR.$$

It is well-known that this class of resonant problems presents special interest since the energy functional of the problem exhibits partial lack of compactness.

Recently Neumann problems (that is, $\beta\equiv 0$) with an indefinite potential were investigated by Mugnai \& Papageorgiou \cite{15} and Papageorgiou \& R\u adulescu \cite{16,17}. Resonant problems were considered by Mugnai \& Papageorgiou \cite{15}, who deal with problems resonant at the first eigenvalue but do not cover the strongly resonant case. Strongly resonant semilinear Dirichlet problems with zero potential, were studied by Landesman \& Lazer \cite{11} (who coined the term ``strong resonance"), Thews \cite{21}, Bartolo, Benci \& Fortunato \cite{2}, Ward \cite{23} (existence of solutions) and Goncalves \& Miyagaki \cite{9} (multiplicity of solutions).

Our approach is based on variational tools coming from the critical point theory and on Morse theory (critical groups). In the next section, for the convenience of the reader, we recall some basic definitions and facts from these theories which we will need in the sequel and we fix our notation.

\section{Mathematical Background-Preliminary Results}

Let $X$ be a Banach space. By $X^*$ we denote the topological dual of $X$ and by $\left\langle \cdot,\cdot\right\rangle$ we denote the duality brackets for the pair $(X^*,X)$. Given $\varphi\in C^1(X,\RR)$, we say that $\varphi$ satisfies the ``Cerami condition at level $c\in\RR$" (the ``$C_c$-condition", for short), if the following property holds:
\begin{center}
``Every sequence $\{u_n\}_{n\geq 1}\subset X$ such that
\begin{eqnarray*}
	&&\varphi(u_n)\rightarrow c\
	\mbox{and}\ (1+||u_n||)\varphi'(u_n)\rightarrow 0\ \mbox{in}\ X^*\ \mbox{as}\ n\rightarrow\infty,
\end{eqnarray*}
admits a strongly convergent subsequence".
\end{center}
\begin{prop}\label{prop1}
	Assume that $\varphi\in C^1(X,\RR)$ is bounded below and let $m=\inf\limits_{X}\varphi$. If $\varphi$ satisfies the $C_c$-condition, then we can find $u_0\in X$ such that $\varphi(u_0)=\inf\limits_{X}\varphi$.
\end{prop}

The next result is known in the literature as the ``second deformation theorem" and is one of the main results in critical point theory. First we introduce some notation. Given $\varphi\in C^1(X,\RR)$ and $c\in\RR$, we define
\begin{eqnarray*}
	&&K_{\varphi}=\{u\in X:\varphi'(u)=0\},\\
	&&K^c_{\varphi}=\{u\in K_{\varphi}:\varphi(u)=c\},\\
	&&\varphi_c=\{u\in X:\varphi(u)\leq c\}.
\end{eqnarray*}
\begin{theorem}\label{th2}
	If $\varphi\in C^1(X,\RR),\ a\in\RR,\ a<b\leq+\infty,\ \varphi$ satisfies the $C_c$-condition for every $c\in\left[a,b\right)$, $\varphi$ has no critical values in $(a,b)$ and $\varphi^{-1}(a)$ contains at most a finite number of critical points, then we can find a deformation $h:[0,1]\times(\varphi^b\backslash K^b_{\varphi})\rightarrow \varphi^b$ such that
	\begin{itemize}
		\item[(a)] $h(1,\varphi^b\backslash K^b_{\varphi})\subset\varphi^a$;
		\item[(b)] $h(t,\cdot)|_{\varphi^a}=id|_{\varphi^a}$ for all $t\in[0,1]$;
		\item[(c)] $\varphi(h(t,u))\leq\varphi(h(s,u))$ for all $t,s\in[0,1]$ with $ s\leq t$, all $u\in\varphi^b\backslash K^b_{\varphi}$ (that is, the deformation $h$ is ``$\varphi$-decreasing").
	\end{itemize}
\end{theorem}
\begin{remark}\label{rem3}
	Note that if $b=+\infty$, then $\varphi^b\backslash K^b_{\varphi}=X$.The conclusion of Theorem \ref{th2} says that $\varphi^a$ is a strong deformation retract of $\varphi^b\backslash K^{b}_{\varphi}$. A special case of this result, is the so-called ``Noncritical Interval Theorem", which says:
	\begin{center}
	``If $\varphi\in C^1(X,\RR)$ satisfies the $C_c$-condition for all $c\in\left[a,b\right]$
	and $K_{\varphi}\cap \varphi^{-1}[a,b]=\emptyset$,
	then $\varphi^a$ is a strong deformation retract of $\varphi^b$".
	\end{center}
\end{remark}

In critical point theory the notion of linking sets, plays a central role:
\begin{definition}\label{def4}
	Let $Y$ be a Hausdorff topological space and $E_0\subseteq E$ and $D$ are nonempty subsets of $Y$. We say that the pair $\{E_0,E\}$ is ``linking" with $D$ in $Y$, if the following conditions hold:
	\begin{itemize}
		\item[(a)] $E_0\cap D=\emptyset$;
		\item[(b)] for any $\gamma\in C(E,Y)$ with $\gamma|_{E_0}=id|_{E_0}$ we have $\gamma(E)\cap D\neq\emptyset$.
	\end{itemize}
\end{definition}

Using this notion, one can prove a general minimax principle from which follow as special cases the classical results of critical point theory (mountain pass theorem, saddle point theorem, generalized mountain pass theorem). For future use we state the mountain pass theorem.
\begin{theorem}\label{th5}
	Assume that $\varphi\in C^1(X,\RR),\ u_0,u_1\in X,\ ||u_1-u_0||>r$
	\begin{eqnarray*}
		&&\max\{\varphi(u_0),\varphi(u_1)\}\leq \inf[\varphi(u):||u-u_0||=r]=m_r,\\
		&&c=\inf\limits_{\gamma\in\Gamma}\max\limits_{0\leq t\leq 1}\varphi(\gamma(t))\ \mbox{with}\ \Gamma=\{\gamma\in C([0,1],X):\gamma(0)=u_0,\gamma(1)=u_1\}
	\end{eqnarray*}
	and $\varphi$ satisfies the $C_c$-condition, then $c\geq m_r$, $c$ is a critical value of $\varphi$ (that is, $K^c_{\varphi}\neq\emptyset$) and if $c=m_r$, then
	$$K^c_{\varphi}\cap B_r(u_0)\neq\emptyset$$
	with $\partial B_r(u_0)=\{u\in X:||u-u_0||=r\}$.
\end{theorem}
\begin{remark}
	For this theorem the linking sets are
	$$E_0=\{u_0,u_1\},\ E=\{(1-t)u_0+tu_1:0\leq t\leq 1\}\ \mbox{and}\ D=\partial B_r(u_0).$$
	
	For details on these and related issues we refer to Gasinski \& Papageorgiou \cite{8}.
	\end{remark}

	In our analysis of problem (\ref{eq1}), we will make use of the following spaces:
	\begin{itemize}
		\item the Sobolev space $W^{1,p}(\Omega),\ 1<p<\infty$;
		\item the Banach space $C^1(\overline{\Omega})$;
		\item the boundary Lebesgue spaces $L^q(\partial\Omega)$, $1\leq q\leq \infty$.
	\end{itemize}
	
	By $||\cdot||$ we denote the norm of $W^{1,p}(\Omega)$ defined by
	$$||u||=[||u||^p_p+||Du||^p_p]^{1/p}\ \mbox{for all}\ u\in W^{1,p}(\Omega).$$

By $\left\langle \cdot,\cdot\right\rangle$ we denote the duality brackets for the pair $(W^{1,p}(\Omega)^*,W^{1,p}(\Omega))$.

The Banach space $C^1(\overline{\Omega})$ is an ordered Banach space with positive (order) cone given by
$$C_+=\{u\in C^1(\overline{\Omega}):u(z)\geq 0\ \mbox{for all}\ z\in\overline{\Omega}\}.$$

This cone has a nonempty interior given by
$$D_+=\{u\in C_+:u(z)>0\ \mbox{for all}\ z\in\overline{\Omega}\}.$$

On $\partial\Omega$ we consider the $(N-1)$-dimensional Hausdorff (surface) measure $\sigma(\cdot)$. Using this measure, we can define in the usual way the Lebesgue spaces $L^q(\partial\Omega),\ 1\leq q\leq\infty$. We know that there exists a unique continuous linear map $\gamma_0:W^{1,p}(\Omega)\rightarrow L^p(\partial\Omega)$, known as the ``trace map", such that
$$\gamma_0(u)=u|_{\partial\Omega}\ \mbox{for all}\ u\in W^{1,p}(\Omega)\cap C(\overline{\Omega}).$$

So, the trace map extends the notion of ``boundary values" to all Sobolev functions. The trace map $\gamma_0$ is compact into $L^q(\partial\Omega)$ with $q\in\left[1,\frac{Np-p}{N-p}\right)$  if $p<N$ and  $q\in\left[1,\infty\right)$ if $p\geq N$. In what follows, for the sake of notational simplicity we drop the use of the map $\gamma_0$. All restrictions of the Sobolev function on $\partial\Omega$ are understood in the sense of traces.

Our hypotheses on the potential function $\xi(\cdot)$ and the boundary coefficient $\beta(\cdot)$ are the following:

\smallskip
$H(\xi):\xi\in L^{\infty}(\Omega)$.

\smallskip
$H(\beta): \beta\in C^{0,\alpha}(\partial\Omega)$ with $\alpha\in(0,1)$ and $\beta(z)\geq 0$ for all $z\in\partial\Omega$.

\smallskip
We consider the $C^1$-functional $\vartheta:W^{1,p}(\Omega)\rightarrow\RR$ defined by
$$\vartheta(u)=||Du||^p_p+\int_{\Omega}\xi(z)|u|^pdz+\int_{\partial\Omega}\beta(z)|u|^pd\sigma\ \mbox{for all}\ u\in W^{1,p}(\Omega).$$

Let $f_0:\Omega\times\RR\rightarrow\RR$ be a Carath\'eodory function such that
$$|f_0(z,x)|\leq a_0(z)(1+|x|^{r-1})\ \mbox{for almost all}\ z\in\Omega,\ \mbox{all}\ x\in\RR,$$
with $a_0\in L^{\infty}(\Omega)_+$ and $1<r\leq p^*=\left\{\begin{array}{ll}
	\frac{Np}{N-p}&\mbox{if}\ p<N\\
	+\infty&\mbox{if}\ p\geq N
\end{array}\right.$ (the critical Sobolev exponent). We set $F_0(z,x)=\int^x_0f_0(z,s)ds$ and consider the $C^1$-functional $\varphi_0:W^{1,p}(\Omega)\rightarrow\RR$ defined by
$$\varphi_0(u)=\frac{1}{p}\vartheta(u)-\int_{\Omega}F_0(z,u)dz\ \mbox{for all}\ u\in W^{1,p}(\Omega).$$

From Papageorgiou \& R\u adulescu \cite{18} (subcritical case) and \cite{19} (critical case) we have the following result.
\begin{prop}\label{prop6}
	Assume that $u_0\in W^{1,p}(\Omega)$ is a local $C^1(\overline{\Omega})$-minimizer of $\varphi_0$, that is, there exists $\rho_0>0$ such that
	$$\varphi_0(u_0)\leq\varphi_0(u_0+h)\ \mbox{for all}\ h\in C^1(\overline{\Omega})\ \mbox{with}\ ||h||_{C^1(\overline{\Omega})}\leq\rho_0.$$
	Then $u_0\in C^{1,\alpha}(\overline{\Omega})$ for some $\alpha\in(0,1)$ and $u_0$ is also a local $W^{1,p}(\Omega)$-minimizer of $\varphi_0$, that is, there exists $\rho_1>0$ such that
	$$\varphi_0(u_0)\leq\varphi_0(u_0+h)\ \mbox{for all}\ h\in W^{1,p}(\Omega)\ \mbox{with}\ ||h||\leq\rho_1.$$
\end{prop}

Let $A:W^{1,p}(\Omega)\rightarrow W^{1,p}(\Omega)^*$ be the nonlinear map defined by
$$\left\langle A(u),h\right\rangle=\int_{\Omega}|Du|^{p-2}(Du,Dh)_{\RR^N}dz\ \mbox{for all}\ u,h\in W^{1,p}(\Omega).$$

The following well-known result summarizes the man properties of the map $A(\cdot)$ (see, for example, Motreanu, Motreanu \& Papageorgiou \cite[p. 40]{14}).
\begin{prop}\label{prop7}
	The map $A:W^{1,p}(\Omega)\rightarrow W^{1,p}(\Omega)^*$ is bounded (maps bounded sets to bounded sets), continuous, monotone (thus maximal monotone too) and of type $(S)_+$ that is,
	$$``u_n\stackrel{w}{\rightarrow}u\ \mbox{in}\ W^{1,p}(\Omega)\ \mbox{and}\ \limsup\limits_{n\rightarrow\infty}\left\langle A(u_n),u_n-u\right\rangle\leq 0\Rightarrow u_n\rightarrow u\ \mbox{in}\ W^{1,p}(\Omega)."$$
\end{prop}

We will also use some facts about the spectrum of the differential operator $u\mapsto -\Delta_pu+\xi(z)u$ with Robin boundary condition.

So, we consider the following nonlinear eigenvalue problem:
\begin{eqnarray}\label{eq2}
	\left\{\begin{array}{ll}
		-\Delta_pu(z)+\xi(z)|u(z)|^{p-2}u(z)=\hat{\lambda}|u(z)|^{p-2}u(z),&\\
		\frac{\partial u}{\partial n_p}+\beta(z)|u|^{p-2}u=0\ \mbox{on}\ \partial\Omega.&
	\end{array}\right\}
\end{eqnarray}

By an eigenvalue, we mean a $\hat{\lambda}\in\RR$ for which problem (\ref{eq2}) has a nontrivial solution $\hat{u}\in W^{1,p}(\Omega)$, known as an eigenfunction corresponding to the eigenvalue $\hat{\lambda}$. From Papageorgiou \& R\u adulescu \cite{19}, we know that $\hat{u}\in L^{\infty}(\Omega)$ and so we can apply Theorem 2 of Lieberman \cite{13} and infer that $\hat{u}\in C^1(\overline{\Omega})$.

From Mugnai \& Papageorgiou \cite{15} and Papageorgiou \& R\u adulescu \cite{18}, we know that problem (\ref{eq2}) admits a smallest eigenvalue $\hat{\lambda}_1\in\RR$ which has the following properties:
\begin{itemize}
	\item $\hat{\lambda}_1$ is isolated in the spectrum $\sigma_0(p)$ of (\ref{eq2}) (that is, we can find $\epsilon>0$ such that $(\hat{\lambda}_1,\hat{\lambda}_1+\epsilon)\cap\sigma_0(p)=\emptyset$).
	\item $\hat{\lambda}_1$ is simple (that is, if $\hat{u},\hat{v}$ are eigenfunctions corresponding to $\hat{\lambda}_1$, then $\hat{u}=\eta\hat{u}$ with $\eta\in\RR\backslash\{0\}$).
	\begin{equation}\label{eq3}
		\bullet\ \hat{\lambda}_1=\inf\left[\frac{\vartheta(u)}{||u||^p_p}:u\in W^{1,p}(\Omega),u\neq 0\right].\hspace{7cm}
	\end{equation}
\end{itemize}

The infimum in (\ref{eq3}) is realized on the one dimensional eigenspace corresponding to $\hat{\lambda}_1$. The above properties of $\hat{\lambda}_1$ imply that the eigenfunctions corresponding to $\hat{\lambda}_1$ do not change sign. Let $\hat{u}_1$ be the $L^p$-normalized (that is, $||\hat{u}_1||_p=1$) positive eigenfunction corresponding to $\hat{\lambda}_1$. As we already mentioned, the nonlinear regularity theory implies that $\hat{u}_1\in C_+$. In fact, the nonlinear maximum principle (see, for example, Gasinski \& Papageorgiou \cite[p. 738]{8}), implies that $\hat{u}_1\in D_+$. An eigenfunction $\hat{u}$ which corresponds to an eigenvalue $\hat{\lambda}\neq\hat{\lambda}_1$ is nodal (that is. sign changing). Since the spectrum $\sigma_0(p)$ of (\ref{eq2}) is closed and $\hat{\lambda}_1$ is isolated, the second eigenvalue $\hat{\lambda}_2$ is well-defined by
\begin{equation}\label{eq4}
	\hat{\lambda}_2=\min[\hat{\lambda}\in\sigma_0(p):\hat{\lambda}>\hat{\lambda}_1].
\end{equation}

To produce additional eigenvalues, we employ the Ljusternik-Schnirelmann minimax scheme, which generates a whole nondecreasing sequence $\{\hat{\lambda}_k\}_{k\in\NN}$ of eigenvalues of (\ref{eq2}) such that $\hat{\lambda}_k\rightarrow+\infty$. These eigenvalues are known as ``variational eigenvalues" and depending on the index used in the execution of the Ljusternik-Schirelmann minimax scheme, we generate different sequences of variational eigenvalues. We do not know if these sequence coincide and if they exhaust the spectrum $\sigma_0(p)$. This is the case if $p=2$ (linear eigenvalue problem) or if $N=1$ (ordinary differential equation). Moreover, we know that all these sequences of variational eigenvalues coincide in the first two elements $\hat{\lambda}_1$ and $\hat{\lambda}_2$, which are given by (\ref{eq3}) and (\ref{eq4}). In fact for $\hat{\lambda}_2$ we have a useful minimax characterization. So, let
\begin{eqnarray*}
	&&\partial B^{L^p}_1=\{u\in L^p(\Omega):||u||_p=1\},\\
	&&M=W^{1,p}(\Omega)\cap \partial B^{L^p}_1,\\
	&&\hat{\Gamma}=\{\hat{\gamma}\in C([-1,1],M):\hat{\gamma}(-1)=-\hat{u}_1,\hat{\gamma}(1)=\hat{u}_1\}.
\end{eqnarray*}

Using these items we can have the following minimax characterization of $\hat{\lambda}_2$ (see \cite{15}, \cite{18}).
\begin{prop}\label{prop8}
	$\hat{\lambda}_2=\inf\limits_{\hat{\gamma}\in\hat{\Gamma}}\max\limits_{-1\leq t\leq 1}\ \vartheta(\hat{\gamma}(t))$.
\end{prop}

Here we use the sequence of variational eigenvalues generated by the Ljusternik-Schni\-rel\-mann scheme when as index we use the Fadell-Rabinowitz cohomological index (see \cite{7}).

Finally let us recall some basic definitions and facts from critical groups which we will use in the sequel. So, let $(Y_1,Y_2)$ be a topological pair such that $Y_2\subseteq Y_1\subseteq X$. By $H_k(Y_1,Y_2)$, $k\in\NN_0$, we denote the $kth$ relative singular homology group with integer coefficients for the pair $(Y_1,Y_2)$. If $\varphi\in C^1(X,\RR)$ and $u\in K^c_{\varphi}$ is isolated, then the critical groups of $\varphi$ at $u$ are defined by
$$C_k(\varphi,u)=H_k(\varphi^c\cap U,\varphi^c\cap U\backslash\{0\})\ \mbox{for all}\ k\in\NN_0,$$
with $U$ being a neighborhood of $u$ such that $K_{\varphi}\cap \varphi^c\cap U=\{u\}$. The excision property of singular homology, implies that this definition is independent of the particular choice of the neighborhood $U$.

Suppose that $\varphi$ satisfies the $C$-condition and $\inf\varphi(K_{\varphi})>-\infty$. Let $c<\inf\varphi(K_{\varphi})$. The critical groups of $\varphi$ at infinity are defined by
$$C_k(\varphi,\infty)=H_k(X,\varphi^c)\ \mbox{for all}\ k\in\NN_0.$$

This definition is independent of the choice of the level $c<\inf\varphi(K_{\varphi})$. Indeed, if $c_0<c<\inf\varphi(K_{\varphi})$, then by the Noncritical Interval Theorem (see Remark \ref{rem3}), we have that
\begin{eqnarray*}
	&&\varphi^{c_0}\ \mbox{is a strong deformation retract of}\ \varphi^c,\\
	&\Rightarrow&H_k(X,\varphi^c)=H_k(X,\varphi^{c_0})\ \mbox{for all}\ k\in\NN_0,\\
	&&(\mbox{see Motreanu, Motreanu \& Papageorgiou \cite[p. 145]{14}}).
\end{eqnarray*}

We introduce the following quantities
\begin{eqnarray*}
	&&M(t,u)=\sum\limits_{k\geq 0}{\rm rank}\, C_k(\varphi,u)t^k\ \mbox{for all}\ t\in\RR,\ \mbox{all}\ u\in K_{\varphi},\\
	&&P(t,\infty)=\sum\limits_{k\geq 0}{\rm rank}\, C_k(\varphi,\infty)t^k\ \mbox{for all}\ t\in\RR\,.
\end{eqnarray*}

Then the ``Morse relation" says that
\begin{equation}\label{eq5}
	\sum\limits_{u\in K_{\varphi}}M(t,u)=P(t,\infty)+(1+t)Q(t)\ \mbox{for all}\ t\in\RR,
\end{equation}
with $Q(t)=\sum\limits_{k\geq 0}\beta_kt^k$ being a formal series in $t\in\RR$ with nonnegative integer coefficients $\beta_k$.

Now let
$$V_p=\{u\in W^{1,p}(\Omega):\int_{\Omega}\hat{u}_1^{p-1}udz=0\}.$$

We have the following direct sum decomposition
$$W^{1,p}(\Omega)=\RR\hat{u}_1\oplus V_p$$

We define
\begin{equation}\label{eq6}
	\hat{\lambda}(p)=\inf\left[\frac{\vartheta(u)}{||u||^p_p}:u\in V_p,u\neq 0\right].
\end{equation}
\begin{prop}\label{prop9}
	If hypotheses $H(\xi),H(\beta)$ hold, then $\hat{\lambda}_1<\hat{\lambda}(p)\leq \hat{\lambda}_2.$
\end{prop}
\begin{proof}
	Note that (\ref{eq3}) and (\ref{eq6}) imply that $\hat{\lambda}_1\leq\hat{\lambda}(p)$. Suppose that $\hat{\lambda}_1=\hat{\lambda}(p)$. Consider a sequence $\{u_n\}_{n\geq 1}\subseteq V_p$ such that
	\begin{equation}\label{eq7}
		||u_n||_p=1\ \mbox{for all}\ n\in\NN\ \mbox{and}\ \vartheta(u_n)\downarrow\hat{\lambda}(p)=\hat{\lambda}_1\ \mbox{as}\ n\rightarrow\infty.
	\end{equation}
	
	So, the sequence $\{u_n\}_{n\geq 1}\subseteq W^{1,p}(\Omega)$ is bounded and thus, by passing to a suitable subsequence if necessary, we may assume that
	\begin{equation}\label{eq8}
		u_n\stackrel{w}{\rightarrow}u\ \mbox{in}\ W^{1,p}(\Omega)\ \mbox{and}\ u_n\rightarrow u\ \mbox{in}\ L^p(\Omega)\ \mbox{and in}\ L^p(\partial\Omega).
	\end{equation}
	
	Since the functional $\vartheta(\cdot)$ is sequentially weakly lower semicontinuous, from (\ref{eq7}) and (\ref{eq8}) it follows that
	\begin{eqnarray}\label{eq9}
		&&\vartheta(u)\leq\hat{\lambda}(p)=\hat{\lambda}_1,\ ||u||_p=1,u\in V_p,\\
		&\Rightarrow&\vartheta(u)=\hat{\lambda}(p)=\hat{\lambda}_1\ (\mbox{see (\ref{eq6})}),\nonumber\\
		&\Rightarrow&u=\eta\hat{u}_1\ \mbox{with}\ \eta\in\RR\backslash\{0\}.\nonumber
	\end{eqnarray}
	
	If $\eta\neq 0$, then $u\not\in V_p$, a contradiction (see (\ref{eq8})).
	
	If $\eta=0$, then $u=0$, a contradiction since $||u||_p=1$ (see (\ref{eq8})).
	
	So, we have proved that $\hat{\lambda}_1<\hat{\lambda}(p).$
	
	Next we show that $\hat{\lambda}(p)\leq\hat{\lambda}_2$. Arguing by contradiction, suppose that $\hat{\lambda}_2<\hat{\lambda}(p)$. Then from Proposition \ref{prop8} we see that we can find $\hat{\gamma}_0\in\hat{\Gamma}$ such that
	\begin{equation}\label{eq10}
		\vartheta(\hat{\gamma}_0(t))<\hat{\lambda}(p)\ \mbox{for all}\ t\in[-1,1].
	\end{equation}
	
	Let $\tau:[-1,1]\rightarrow\RR$ be defined by
	$$\tau(t)=\int_{\Omega}\hat{u}_1^{p-1}\hat{\gamma}_0(t)dz\ \mbox{for all}\ t\in[-1,1].$$
	
	Evidently $\tau(\cdot)$ is continuous and we have
	$$\tau(-1)=-1,\ \tau(1)=1\ (\mbox{recall}\ ||\hat{u}_1||_p=1).$$
	
	So, by Bolzano's theorem, we can find $t_0\in(-1,1)$ such that
	\begin{eqnarray}\label{eq11}
		&&\tau(t_0)=\int_{\Omega}\hat{u}_1^{p-1}\hat{\gamma}_0(t_0)dz=0,\nonumber\\
		&\Rightarrow&\hat{\gamma}_0(t_0)\in V_p\nonumber\\
		&\Rightarrow&\hat{\lambda}(p)\leq\vartheta(\hat{\gamma}_0(t_0))\ (\mbox{see (\ref{eq6})}).
	\end{eqnarray}
	
	Comparing (\ref{eq10}) and (\ref{eq11}), we reach a contradiction. Therefore we obtain
	$$\hat{\lambda}_1<\hat{\lambda}(p)\leq\hat{\lambda}_2.$$
The proof is now complete.
\end{proof}
\begin{remark}
	If $p=2$, then $\hat{\lambda}(2)=\hat{\lambda}_2.$
\end{remark}

Now let $\lambda>\hat{\lambda}_2,\ \lambda\not\in\sigma_0(p)$ and consider the $C^1$-functional $\psi_{\lambda}:W^{1,p}(\Omega)\rightarrow\RR$ defined by
$$\psi_{\lambda}(u)=\frac{1}{p}\vartheta(u)-\frac{\lambda}{p}||u||^p_p\ \mbox{for all}\ u\in W^{1,p}(\Omega).$$
\begin{prop}\label{prop10}
	If hypotheses $H(\xi),H(\beta)$ hold, then $C_0(\psi_{\lambda},0)=C_1(\psi_{\lambda},0)=0$.
\end{prop}
\begin{proof}
	Let $D=\{u\in W^{1,p}(\Omega):\vartheta(u)<\lambda||u||^p_p\}$. Evidently $D\subseteq W^{1,p}(\Omega)$ is open and $\pm\hat{u}_1\in D$.
	\begin{claim}
		The set $D$ is path-connected.
	\end{claim}
	
	Let $u\in D$ and let $E_u$ be the path-component of $D$ containing $u$. We set
	\begin{equation}\label{eq12}
		m_u=\inf\left[\frac{\vartheta(v)}{||v||^p_p}:v\in E_u\right]<\lambda .
	\end{equation}
	
	Let $\{v_n\}_{n\geq 1}\subseteq E_u$ be such that
	\begin{equation}\label{eq13}
		\frac{\vartheta(v_n)}{||v_n||^p_p}\downarrow m_u\ \mbox{as}\ n\rightarrow\infty.
	\end{equation}
	
	The $p$-homogenicity of $\vartheta(\cdot)$, allows us to assume that
	\begin{equation}\label{eq14}
		||v_n||_p=1\ \mbox{for all}\ n\in\NN.
	\end{equation}
	
	Then from (\ref{eq13}) and (\ref{eq14}) it follows that $\{v_n\}_{n\geq 1}\subseteq W^{1,p}(\Omega)$ is bounded. Recall that
	$$M=W^{1,p}(\Omega)\cap\partial B^{L^p}_1=\{u\in W^{1,p}(\Omega):||u||_p=1\}.$$
	
	Employing the Ekeland variational principle (see, for example, Gasinski \& Papageorgiou \cite[p. 579]{8}), we can find $\{y_n\}_{n\geq 1}\subset\overline{E_u\cap M}$ such that
	\begin{eqnarray}\label{eq15}
		\left\{\begin{array}{ll}
			\vartheta(y_n)\leq\vartheta(v_n)\leq m_u+\frac{1}{n^2},\ ||y_n-v_n||<\frac{1}{n},&\\
			\vartheta(y_n)\leq\vartheta(v)+\frac{1}{n^2}||v-y_n||\ \mbox{for all}\ v\in\overline{E_u\cap M},\ \mbox{all}\ n\in\NN.&
		\end{array}\right\}
	\end{eqnarray}
	
	Suppose that $y_n\in\partial\overline{(E_u\cap M)}$ for infinitely many $n\in\NN$ (to simplify things we assume that it holds for every $n\in\NN$). Then Lemma 3.5(iii) of Cuesta, de Figueiredo \& Gossez \cite{5} implies that
	$$\vartheta(y_n)=\lambda\leq\vartheta(v_n)\leq m_u+\frac{1}{n^2}<\lambda\ \mbox{for all}\ n\in\NN\ \mbox{big (see (\ref{eq15}))},$$
	a contradiction. This means that
	$$y_n\in E_u\cap M\ \mbox{for all}\ n\in\NN.$$
	
	Then from (\ref{eq15}) it follows that
	\begin{equation}\label{eq16}
		\left(\vartheta|_{M}\right)'(y_n)\rightarrow 0\ \mbox{as}\ n\rightarrow\infty .
	\end{equation}
	
	As in the proof of Proposition 5 of Papageorgiou \& R\u adulescu \cite{18} (see Claim 1 in that proof), we can see that
	\begin{equation}\label{eq17}
		\vartheta|_M\ \mbox{satisfies the C-condition}.
	\end{equation}
	
	Then from (\ref{eq16}), (\ref{eq17}) and by passing to a suitable subsequence if necessary, we see that we may assume that
	\begin{eqnarray}\label{eq18}
		&&y_n\rightarrow y\ \mbox{in}\ W^{1,p}(\Omega)\ \mbox{as}\ n\rightarrow \infty\\
		&\Rightarrow&y\in\overline{E_u\cap M}\ \mbox{and}\ \vartheta(y)=m_u<\lambda,\nonumber\\
		&\Rightarrow&y\in E_u\cap M\ (\mbox{as before using Lemma 3.5(iii) in Cuesta, de Figueiredo \& Gossez \cite{5}}).\nonumber
	\end{eqnarray}
	
	Hence to prove the claim, it suffices to connect $y$ and $\hat{u}_1$ with a path staying in $D$ (see Dugundji \cite[p. 115]{6}). First suppose that $y\leq 0$. Then $y=-\hat{u}_1$ (see (\ref{eq3})) and the desired path is provided by Proposition \ref{prop8} (recall $\lambda>\hat{\lambda}_2$). Now suppose that $y\geq 0$, then $y=\hat{u}_1$. Therefore, we may assume that
	$$y^+,\ y^-\neq 0.$$
	
	We set
	$$e_t=\frac{y^+-(1-t)y^-}{||y^+-(1-t)y^-||_p}\in M\ \mbox{for all}\ t\in[0,1].$$
	
	From (\ref{eq16}) and (\ref{eq18}), we have
	\begin{eqnarray}\label{eq19}
		&&\left\langle A(y),h\right\rangle+\int_{\Omega}\xi(z)|y|^{p-2}yhdz+\int_{\partial\Omega}\beta(z)|y|^{p-2}yhd\sigma=m_u\int_{\Omega}|y|^{p-2}yhdz\\
		&&\mbox{for all}\ h\in W^{1,p}(\Omega).\nonumber
	\end{eqnarray}
	
	In (\ref{eq19}) first we choose $h=y^+\in W^{1,p}(\Omega)$ and then we choose $h=-y^-\in W^{1,p}(\Omega)$. We obtain
	$$\vartheta(y^+)=m_u||y^+||^p_p\ \mbox{and}\ \vartheta(y^-)=m_u||y^-||^p_p.$$
	
	Since $y^+$ and $y^-$ have disjoint interior supports, it follows that
	\begin{eqnarray*}
		&&\vartheta(e_t)=m_u||e_t||^p_p=m_u\ \mbox{for all}\ t\in[0,1],\\
		&\Rightarrow&e_t\in D\ \mbox{for all}\ t\in[0,1].
	\end{eqnarray*}
	
	Note that
	$$e_0=\frac{y}{||y||_p}=y\ (\mbox{recall}\ y\in E_u\cap M)\ \mbox{and}\ e_1=\frac{y^+}{||y^+||_p}=\hat{u}_1.$$
	
	Therefore $t\mapsto e_t$ is the desired path in $D$. This proves the claim.
	
	If $e\in D$, then from the claim we have
	\begin{equation}\label{eq20}
		H_0(D,e)=0\ (\mbox{see Motreanu, Motreanu \& Papageorgiou \cite[p. 152]{14}}).
	\end{equation}
	
	Consider the $0$-sublevel set of $\psi_{\lambda}$
	$$\psi^0_{\lambda}=\{u\in W^{1,p}(\Omega):\psi_{\lambda}(u)\leq 0\}.$$
	
	Since $\psi_{\lambda}(\cdot)$ is $p$-homogeneous, it follows that
	\begin{eqnarray}\label{eq21}
		&&\psi^0_{\lambda}\ \mbox{is contractible},\nonumber\\
		&\Rightarrow&H_k(\psi^0_{\lambda},e)=0,\ \forall\, k\in\NN_0\ (\mbox{see Motreanu, Motreanu \& Papageorgiou \cite[p. 147]{14}}).
	\end{eqnarray}
	
	Let $\epsilon>0$ be small. Theorem \ref{th2} (the second deformation theorem) implies that
	\begin{equation}\label{eq22}
		\psi^0_{\lambda}\backslash\{0\}\ \mbox{and}\ \psi^{-\epsilon}_{\lambda}\ \mbox{are homotopy equivalent.}
	\end{equation}
	
	Also, let
	$$\dot{\psi}_{\lambda}^0=\{u\in W^{1,p}(\Omega):\psi_{\lambda}(u)<0\}=D.$$
	
	Since $K_{\psi_{\lambda}}=\{0\}$ (recall that $\lambda\not\in\sigma_0(p)$), from Granas \& Dugundji \cite[p. 407]{10},  we have
	\begin{equation}\label{eq23}
		\dot{\psi}_{\lambda}^0=D\ \mbox{and}\ \psi^{-\epsilon}_{\lambda}\ \mbox{are homotopy equivalent}.
	\end{equation}
	
	From (\ref{eq22}) and (\ref{eq23}) we infer that
	\begin{eqnarray}\label{eq24}
		&&\psi^0_{\lambda}\backslash\{0\}\ \mbox{and}\ D\ \mbox{are homotopy equivalent},\nonumber\\
		&\Rightarrow&H_k(\psi^0_{\lambda}\backslash\{0\},e)=H_k(D,e)\ \mbox{for all}\ k\in\NN_0,\nonumber\\
		&\Rightarrow&H_0(\psi^0_{\lambda}\backslash\{0\},e)=0\  (\mbox{see (\ref{eq20})}).
	\end{eqnarray}
	
	We consider the following long exact reduced singular homology sequence
	\begin{eqnarray}\label{eq25}
		&&\ldots\rightarrow H_k(\psi^0_{\lambda}\backslash\{0\},e)\rightarrow H_k(\psi^0_{\lambda},e)=0\xrightarrow{i_*}H_k(\psi^0_{\lambda},\psi^0_{\lambda}\backslash\{0\})=C_k(\psi_{\lambda},0)\nonumber\\
		&&\hspace{3cm}\xrightarrow{\partial_*}H_{k-1}(\psi^0_{\lambda}\backslash\{0\},e)\rightarrow\cdots
	\end{eqnarray}
	with $i_*$ being the group homomorphism corresponding to the inclusion map $i$ and $\partial_*$ is the boundary isomomorphism. From the exactness of (\ref{eq25}) we have that $\partial_*$ is a homomorphism between $C_1(\psi_{\lambda},0)$ and a subgroup of $H_0(\psi^0_{\lambda}\backslash\{0\},e)$ and so
	\begin{equation}\label{eq26}
		C_1(\psi_{\lambda},0)=0\ (\mbox{see (\ref{eq24})}).
	\end{equation}
	
	From (\ref{eq25}) and (\ref{eq26}) it follows that
	$$C_0(\psi_{\lambda},0)=0.$$
The proof is now complete.
\end{proof}

\section{Resonance at a Nonprincipal Eigenvalue}

In this section we prove two existence theorems when the equation is resonant with respect to a nontrivial eigenvalue $\hat{\lambda}_m$.

For the first existence theorem, the hypotheses on the reaction term are the following:

\smallskip
$H_1:$ $f:\Omega\times\RR\rightarrow\RR$ is a Carath\'eodory function such that $f(z,0)=0$ for almost all $z\in\Omega$ and
\begin{itemize}
	\item[(i)] for every $\rho>0$, there exists $a_{\rho}\in L^{\infty}(\Omega)_+$ such that
	$$|f(z,x)|\leq a_{\rho}(z)\ \mbox{for almost all}\ z\in\Omega,\ \mbox{all}\ |x|\leq\rho;$$
	\item[(ii)] there exists $m\in\NN,\ m\geq 2$ such that
	$$\lim\limits_{x\rightarrow\pm\infty}\frac{f(z,x)}{|x|^{p-2}x}=\hat{\lambda}_m\ \mbox{uniformly for almost all}\ z\in\Omega;$$
	\item[(iii)] if $F(z,x)=\int^x_0f(z,s)ds$, then $\lim\limits_{x\rightarrow\pm\infty}[f(z,x)x-pF(z,x)]=+\infty$ uniformly for almost all $z\in\Omega$;
	\item[(iv)] there exist $\eta\in L^{\infty}(\Omega),\ \hat{\lambda}_1\leq\eta(z)$ for almost all $z\in\Omega$, $\eta\not\equiv\hat{\lambda}_1,\ \hat{\eta}<\hat{\lambda}(p)$ and $\delta>0$ such that
	$$\frac{1}{p}\eta(z)|x|^p\leq F(z,x)\leq\frac{1}{p}\hat{\eta}|x|^p\ \mbox{for almost all}\ z\in\Omega,\ \mbox{all}\ |x|\leq\delta.$$
\end{itemize}

Let $\varphi:W^{1,p}(\Omega)\rightarrow\RR$ be the energy (Euler) functional for problem (\ref{eq1}) defined by
$$\varphi(u)=\frac{1}{p}\vartheta(u)-\int_{\Omega}F(z,u)dz\ \mbox{for all}\ u\in W^{1,p}(\Omega).$$

Evidently $\varphi\in C^1(W^{1,p}(\Omega))$.
\begin{prop}\label{prop11}
	If hypotheses $H(\xi),H(\beta),H_1$ hold, then the functional $\varphi$ satisfies the C-condition.
\end{prop}
\begin{proof}
	Let $\{u_n\}_{n\geq 1}\subseteq W^{1,p}(\Omega)$ be a sequence such that
	\begin{eqnarray}
		&&|\varphi(u_n)|\leq M_1\ \mbox{for some}\ M_1>0,\ \mbox{all}\ n\in\NN,\label{eq27}\\
		&&(1+||u_n||)\varphi'(u_n)\rightarrow 0\ \mbox{in}\ W^{1,p}(\Omega)^*\ \mbox{as}\ n\rightarrow\infty\,.\label{eq28}
	\end{eqnarray}
	
	From (\ref{eq28}) we have
	\begin{eqnarray}\label{eq29}
		&&\left|\left\langle A(u_n),h\right\rangle+\int_{\Omega}\xi(z)|u_n|^{p-2}u_nhdz+\int_{\partial\Omega}\beta(z)|u_n|^{p-2}u_nhd\sigma-\int_{\Omega}f(z,u_n)hdz\right|\nonumber\\
		&&\leq\frac{\epsilon_n||h||}{1+||u_n||}\ \mbox{for all}\ h\in W^{1,p}(\Omega),\ \mbox{with}\ \epsilon_n\rightarrow 0^+.
	\end{eqnarray}
	
	In (\ref{eq29}) we choose $h=u_n\in W^{1,p}(\Omega)$. Then
	\begin{equation}\label{eq30}
		-\vartheta(u_n)+\int_{\Omega}f(z,u_n)u_ndz\leq\epsilon_n\ \mbox{for all}\ n\in\NN.
	\end{equation}
	
	From (\ref{eq27}) we have
	\begin{equation}\label{eq31}
		\vartheta(u_n)-\int_{\Omega}pF(z,u_n)dz\leq pM_1\ \mbox{for all}\ n\in\NN.
	\end{equation}
	
	Adding (\ref{eq30}) and (\ref{eq31}), we obtain
	\begin{equation}\label{eq32}
		\int_{\Omega}[f(z,u_n)u_n-pF(z,u_n)]dz\leq M_2\ \mbox{for some}\ M_2>0,\ \mbox{all}\ n\in\NN.
	\end{equation}
	\begin{claim}
		$\{u_n\}_{n\geq 1}\subseteq W^{1,p}(\Omega)$ is bounded.
	\end{claim}
	
	Arguing by contradiction, suppose that the claim is not true. By passing to a suitable subsequence if necessary, we can say that
	\begin{equation}\label{eq33}
		||u_n||\rightarrow\infty.
	\end{equation}
	
	Let $y_n=\frac{u_n}{||u_n||}$ for all $n\in\NN$. Then $||y_n||=1$ and so we may assume that
	\begin{equation}\label{eq34}
		y_n\stackrel{w}{\rightarrow}y\ \mbox{in}\ W^{1,p}(\Omega)\ \mbox{and}\ y_n\rightarrow y\ \mbox{in}\ L^p(\Omega)\ \mbox{and in}\ L^p(\partial\Omega).
	\end{equation}
	
	From (\ref{eq29}) we have
	\begin{eqnarray}\label{eq35}
		&&\left|\left\langle A(y_n),h\right\rangle+\int_{\Omega}\xi(z)|y_n|^{p-2}y_nhdz+\int_{\partial\Omega}\beta(z)|y_n|^{p-2}y_nhd\sigma-\int_{\Omega}\frac{f(z,u_n)}{||u_n||^{p-1}}hdz\right|\\
		&&\leq\frac{\epsilon_n||h||}{(1+||u_n||)||u_n||^{p-1}}\ \mbox{for all}\ n\in\NN\nonumber.
	\end{eqnarray}
	
	From hypotheses $H_1(i),(ii)$ we see that
	\begin{eqnarray}\label{eq36}
		&&|f(z,x)|\leq c_1(1+|x|^{p-1})\ \mbox{for almost all}\ z\in\Omega,\ \mbox{all}\ x\in\RR,\ \mbox{with}\ c_1>0,\nonumber\\
		&\Rightarrow&\left\{\frac{f(\cdot,u_n(\cdot))}{||u_n||^{p-1}}\right\}_{n\geq 1}\subseteq L^{p'}(\Omega)\left(\frac{1}{p}+\frac{1}{p'}=1\right)\ \mbox{is bounded}.
	\end{eqnarray}
	
	In (\ref{eq35}) we choose $h=y_n-y\in W^{1,p}(\Omega)$, pass to the limit as $n\rightarrow\infty$ and use (\ref{eq33}), (\ref{eq34}), (\ref{eq36}). We obtain
	\begin{eqnarray}\label{eq37}
		&&\lim\limits_{n\rightarrow\infty}\left\langle A(y_n),y_n-y\right\rangle=0,\nonumber\\
		&\Rightarrow&y_n\rightarrow y\ \mbox{in}\ W^{1,p}(\Omega)\ \mbox{(see Proposition \ref{prop7}) and so}\ ||y||=1.
	\end{eqnarray}
	
	From (\ref{eq37}) we see that $y\neq 0$. Let $E=\{z\in\Omega:y(z)\neq 0\}$. If by $|\cdot|_N$ we denote the Lebesgue measure on $\RR^N$, then $|E|_N>0$. We have
	\begin{eqnarray}\label{eq38}
		&&|u_n(z)|\rightarrow+\infty\ \mbox{for almost all}\ z\in E,\nonumber\\
		&\Rightarrow&f(z,u_n(z))u_n(z)-pF(z,u_n(z))\rightarrow+\infty\ \mbox{for almost all}\ z\in\Omega\\
		&&(\mbox{see hypothesis}\ H_1(iii)).\nonumber
	\end{eqnarray}
	
	From (\ref{eq38}), hypothesis $H_1(iii)$ and Fatou's lemma, we have
	\begin{equation}\label{eq39}
		\int_E[f(z,u_n)u_n-pF(z,u_n)]dz\rightarrow+\infty\ \mbox{as}\ n\rightarrow\infty.
	\end{equation}
	
	On the other hand, hypotheses $H_1(i),(iii)$ imply that we can find $c_2>0$ such that
	\begin{equation}\label{eq40}
		-c_2\leq f(z,x)x-pF(z,x)\ \mbox{for almost all}\ z\in\Omega,\ \mbox{all}\ x\in\RR.
	\end{equation}
	
	Then we have
	\begin{eqnarray*}
		&&\int_{\Omega}[f(z,u_n)u_n-pF(z,u_n)]dz\\
		&=&\int_E[f(z,u_n)u_n-pF(z,u_n)]dz+\int_{\Omega\backslash E}[f(z,u_n)u_n-pF(z,u_n)]dz\\
		&\geq&\int_{E}[f(z,u_n)u_n-pF(z,u_n)]dz-c_2|\Omega|_N\ (\mbox{see (\ref{eq40})}),\\
		\Rightarrow&&\int_{\Omega}[f(z,u_n)u_n-pF(z,u_n)]dz\rightarrow+\infty\ \mbox{as}\ n\rightarrow\infty\ (\mbox{see (\ref{eq39})}).
	\end{eqnarray*}
	
	This contradicts (\ref{eq32}). So, we have proved the claim.
	
	Because of the claim, at least for a subsequence, we may assume that
	\begin{equation}\label{eq41}
		u_n\stackrel{w}{\rightarrow}u\ \mbox{in}\ W^{1,p}(\Omega)\ \mbox{and}\ u_n\rightarrow u\ \mbox{in}\ L^p(\Omega)\ \mbox{and}\ L^p(\partial\Omega).
	\end{equation}
	
	Note that $\{f(\cdot,u_n(\cdot))\}_{n\geq 1}\subseteq L^{p'}(\Omega)$ is bounded. So, if in (\ref{eq29}) we choose $h=u_n-u\in W^{1,p}(\Omega)$, pass to the limit as $n\rightarrow\infty$ and use (\ref{eq41}), then
	\begin{eqnarray*}
		&&\lim\limits_{n\rightarrow\infty}\left\langle A(u_n),u_n-u\right\rangle=0,\\
		&\Rightarrow&u_n\rightarrow u\ \mbox{in}\ W^{1,p}(\Omega)\ (\mbox{see Proposition \ref{prop7}}),\\
		&\Rightarrow&\varphi\ \mbox{satisfies the C-condition}.
	\end{eqnarray*}
The proof is now complete.
\end{proof}

\begin{prop}\label{prop12}
	If hypotheses $H(\xi),H(\beta),H_1$ hold, then $C_1(\varphi,0)\neq 0$.
\end{prop}
\begin{proof}
	We consider the direct sum decomposition
	$$W^{1,p}(\Omega)=\RR\hat{u}_1\oplus V_p.$$
	
	For $|t|\in(0,1)$ small we have
	\begin{equation}\label{eq42}
		|t|\hat{u}_1(z)\in\left(0,\delta\right]\ \mbox{for all}\ z\in\overline{\Omega}\ (\mbox{recall}\ \hat{u}_1\in D_+).
	\end{equation}
	
	Here $\delta>0$ is as in hypothesis $H_1(iv)$. Then
	\begin{eqnarray}\label{eq43}
		\varphi(t\hat{u}_1)&=&\frac{|t|^p}{p}\vartheta(\hat{u}_1)-\int_{\Omega}F(z,t\hat{u}_1)dz\nonumber\\
		&\leq&\frac{|t|^p}{p}\left[\vartheta(\hat{u}_1)-\int_{\Omega}\eta(z)\hat{u}^p_1dz\right]\nonumber\\
		&&(\mbox{see hypothesis}\ H_1(iv)\ \mbox{and recall that}\ ||\hat{u}_1||_p=1)\nonumber\\
		&=&\frac{|t|^p}{p}\int_{\Omega}[\hat{\lambda}_1-\eta(z)]\hat{u}^p_1dz\nonumber\\
		&<&0\ \mbox{for all}\ |t|\in(0,1)\ \mbox{small}\ (\mbox{recall that}\ \hat{u}_1\in D_+).
	\end{eqnarray}
	
	Hypotheses $H_1$ imply that given $r\in(p,p^*)$, we can find $c_3=c_3(r)>0$ such that
	\begin{equation}\label{eq44}
		F(z,x)\leq\frac{\hat{\eta}}{p}|x|^p+c_3|x|^r\ \mbox{for almost all}\ z\in\Omega,\ \mbox{all}\ x\in\RR.
	\end{equation}
	
	Then for all $v\in V_p$, we have
	\begin{eqnarray*}
		\varphi(v)&=&\frac{1}{p}\vartheta(v)-\int_{\Omega}F(z,v)dz\\
		&\geq&\frac{1}{p}[\vartheta(v)-\hat{\eta}||v||^p_p]-c_4||v||^r\ \mbox{for some}\ c_4>0\ (\mbox{see (\ref{eq44})})\\
		&\geq&c_5||v||^p-c_4||v||^r\ \mbox{for some}\ c_5>0\ (\mbox{recall that}\ \hat{\eta}<\hat{\lambda}(p)).
	\end{eqnarray*}
	
	Since $p<r$, we can find $\delta_1\in(0,1)$ small such that
	\begin{equation}\label{eq45}
		\varphi(v)>0=\varphi(0)\ \mbox{for all}\ 0<||v||\leq\delta_1.
	\end{equation}
	
	Relations (\ref{eq43}) and (\ref{eq45}) imply that $\varphi$ has a local linking at the origin with respect to the decomposition $\RR\hat{u}_1\oplus V_p$. So, Corollary 6.88 of Motreanu, Motreanu \& Papageorgiou \cite[p. 172]{14} implies that
	$C_1(\varphi,0)\neq 0.$
\end{proof}
\begin{prop}\label{prop13}
	If hypotheses $H(\xi),H(\beta),H_1$ hold, then $C_0(\varphi,\infty)=C_1(\varphi,\infty)=0$ and $C_m(\varphi,\infty)\neq 0$.
\end{prop}
\begin{proof}
	Let $\lambda\in(\hat{\lambda}_m,\hat{\lambda}_{m+1})\backslash\sigma_0(p)$ and as before (see Section 2), let $\psi_{\lambda}:W^{1,p}(\Omega)\rightarrow \RR$ be the $C^1$-functional defined by
	$$\psi_{\lambda}(u)=\frac{1}{p}\vartheta(u)-\frac{\lambda}{p}||u||^p_p\ \mbox{for all}\ u\in W^{1,p}(\Omega).$$
	
	We consider the homotopy $h(t,u)$ defined by
	$$h(t,u)=(1-t)\varphi(u)+t\psi_{\lambda}(u)\ \mbox{for all}\ t\in[0,1],\ \mbox{all}\ u\in W^{1,p}(\Omega).$$
	\begin{claim}
		There exist $k_0\in\RR$ and $\delta_0>0$ such that
		$$h(t,u)\leq k_0\Rightarrow(1+||u||)||h'_u(t,u)||_*\geq\delta_0\ \mbox{for all}\ t\in[0,1].$$
	\end{claim}
	
		We argue indirectly. So, suppose that the claim is not true. Note that the homotopy $h(t,u)$ maps bounded sets into bounded sets. So, we can find $\{t_n\}_{n\geq 1}\subseteq[0,1]$ and $\{u_n\}_{n\geq 1}\subseteq W^{1,p}(\Omega)$ such that
		\begin{equation}\label{eq46}
			t_n\rightarrow t,||u_n||\rightarrow\infty,h(t_n,u_n)\rightarrow-\infty\ \mbox{and}\ (1+||u_n||)h'_u(t_n,u_n)\rightarrow 0\ \mbox{in}\ W^{1,p}(\Omega)^*.
		\end{equation}
		
		From the last convergence in (\ref{eq46}), we have
		\begin{eqnarray}\label{eq47}
			&&\left|\left\langle A(u_n),h\right\rangle+\int_{\Omega}\xi(z)|u_n|^{p-2}u_nhdz+\int_{\partial\Omega}\beta(z)|u_n|^{p-2}u_nhd\sigma-(1-t_n)\int_{\Omega}f(z,u_n)hdz-\right.\nonumber\\
			&&\quad\left.t_n\lambda\int_{\Omega}|u_n|^{p-2}u_nhdz\right|\leq\frac{\epsilon_n||h||}{1+||u_n||}
		\end{eqnarray}
		for all $h\in W^{1,p}(\Omega)$, with $\epsilon_n\rightarrow 0^+.$
		
		Let $y_n=\frac{u_n}{||u_n||},\ n\in\NN$. Then $||y_n||=1$ for all $n\in\NN$ and so we may assume that
		\begin{equation}\label{eq48}
			y_n\stackrel{w}{\rightarrow}y\ \mbox{in}\ W^{1,p}(\Omega)\ \mbox{and}\ y_n\rightarrow y\ \mbox{in}\ L^p(\Omega)\ \mbox{and}\ L^p(\partial\Omega).
		\end{equation}
		
		From (\ref{eq47}) we have
		\begin{eqnarray}\label{eq49}
			&&\left|\left\langle A(y_n),h\right\rangle+\int_{\Omega}\xi(z)|y_n|^{p-2}y_nhdz+\int_{\partial\Omega}\beta(z)|y_n|^{p-2}y_nhd\sigma-(1-t_n)\int_{\Omega}\frac{f(z,u_n)}{||u_n||^{p-1}}hdz-\right.\nonumber\\
			&&\quad\left.t_n\lambda\int_{\Omega}|y_n|^{p-2}y_nhdz\right|\leq\frac{\epsilon_n||h||}{1+||u_n||}\ \mbox{for all}\ n\in\NN\,.
		\end{eqnarray}
		
		Recall that
		\begin{equation}\label{eq50}
			\left\{\frac{f(\cdot,u_n(\cdot))}{||u_n||^{p-1}}\right\}_{n\geq 1}\subseteq L^{p'}(\Omega)\ \mbox{is bounded}\ (\mbox{see hypotheses}\ H_1(i),(ii)).
		\end{equation}
		
		Hypothesis $H_1(ii)$ and (\ref{eq50}) imply that at least for a subsequence, we have
		\begin{equation}\label{eq51}
			\frac{f(\cdot,u_n(\cdot))}{||u_n||^{p-1}}\stackrel{w}{\rightarrow}\hat{\lambda}_m|y|^{p-2}y\ \mbox{in}\ L^{p'}(\Omega)
		\end{equation}
		(see Aizicovici, Papageorgiou \& Staicu \cite{1}, proof of Proposition 30).
		
		In (\ref{eq49}) we choose $h=y_n-y\in W^{1,p}(\Omega)$, pass to the limit as $n\rightarrow\infty$ and use (\ref{eq48}), (\ref{eq51}). Then
		\begin{eqnarray}\label{eq52}
			&&\lim\limits_{n\rightarrow\infty}\left\langle A(y_n),y_n-y\right\rangle=0,\nonumber\\
			&\Rightarrow&y_n\rightarrow y\ \mbox{in}\ W^{1,p}(\Omega)\ \mbox{(see Proposition \ref{prop7}), hence}\ ||y||=1.
		\end{eqnarray}
		
		So, if in (\ref{eq49}) we pass to the limit as $n\rightarrow\infty$ and use (\ref{eq51}) and (\ref{eq52}), then
		\begin{eqnarray}\label{eq53}
			&&\left\langle A(y),h\right\rangle+\int_{\Omega}\xi(z)|y|^{p-2}yhdz+\int_{\partial\Omega}\beta(z)|y|^{p-2}yhd\sigma=\lambda_t\int_{\Omega}|y|^{p-2}yhdz\nonumber\\
			&&\mbox{for all}\ h\in W^{1,p}(\Omega)\ \mbox{with}\ \lambda_t=(1-t)\hat{\lambda}_m+t\lambda,\nonumber\\
			&\Rightarrow&-\Delta_py(z)+\xi(z)|y(z)|^{p-2}y(z)=\lambda_t|y(z)|^{p-2}y(z)\ \mbox{for almost all}\ z\in\Omega,\nonumber\\
			&&\frac{\partial y}{\partial n_p}+\beta(z)|y|^{p-2}y=0\ \mbox{on}\ \partial\Omega\ (\mbox{see Papageorgiou \& R\u adulescu \cite{18}}).
		\end{eqnarray}
		
		If $\lambda_t\not\in\sigma_0(p)$, then from (\ref{eq53}) it follows that
		$$y=0,\ \mbox{a contradiction (see (\ref{eq52}))}.$$
		
		So, suppose that $\lambda_t\in\sigma_0(p)$. If $D=\{z\in\Omega:y(z)\neq 0\}$, then from (\ref{eq52}) we see that
		$$|D|_N>0\ \mbox{and}\ |u_n(z)|\rightarrow\infty\ \mbox{for almost all}\ z\in D.$$
		
		Reasoning as in the proof of Proposition \ref{prop11}, we show that
		\begin{equation}\label{eq54}
			\int_{\Omega}[f(z,u_n)u_n-pF(z,u_n)]dz\rightarrow+\infty\ \mbox{as}\ n\rightarrow\infty.
		\end{equation}
		
		From the third convergence in (\ref{eq46}), we have
		\begin{equation}\label{eq55}
			\vartheta(u_n)-(1-t_n)\int_{\Omega}pF(z,u_n)dz-t_n\lambda||u_n||^p_p\leq -1\ \mbox{for all}\ n\geq n_0.
		\end{equation}
		
		In (\ref{eq47}) we choose $h=u_n\in W^{1,p}(\Omega)$. Then
		\begin{equation}\label{eq56}
			-\vartheta(u_n)+(1-t_n)\int_{\Omega}f(z,u_n)u_ndz+t_n\lambda||u_n||^p_p\leq\epsilon_n\ \mbox{for all}\ n\in\NN
		\end{equation}
		
		By choosing $n_0\in\NN$ even bigger if necessary, we can have
		$$\epsilon_n\in(0,1)\ \mbox{for all}\ n\geq n_0.$$
		
		Adding (\ref{eq55}) and (\ref{eq56}), we obtain
		\begin{equation}\label{eq57}
			(1-t_n)\int_{\Omega}[f(z,u_n)u_n-pF(z,u_n)]dz\leq 0\ \mbox{for all}\ n\geq n_0.
		\end{equation}
		
		We may assume that $t_n\in\left[0,1\right)$ for all $n\geq n_0$. Otherwise, there exists a subsequence $\{t_{n_k}\}_{k\geq 1}$ of $\{t_n\}_{n\geq 1}$ with $t_{n_k}=1$ for all $k\in\NN$. Hence $t=1$ and so $\lambda_t=\lambda\not\in\sigma_0(p)$, a contradiction (recall that we  have assumed that $\lambda_t\in\sigma_0(p)$). Therefore $t_n\in\left[0,1\right)$ for all $n\geq n_0$ and so from (\ref{eq57}), we have
		\begin{equation}\label{eq58}
			\int_{\Omega}[f(z,u_n)u_n-pF(z,u_n)]dz\leq 0\ \mbox{for all}\ n\geq n_0.
		\end{equation}
		
		Comparing (\ref{eq54}) and (\ref{eq58}) we have a contradiction. This proves the claim.
		
		Note that the above argument also shows that for every $t\in[0,1],\ h(t,\cdot)$ satisfies the C-condition. We apply Theorem 5.1.21 of Chang \cite[p. 334]{4} (see also  Liang \& Su \cite[Proposition 3.2]{12}) and have
		\begin{eqnarray}\label{eq59}
			&&C_k(h(0,\cdot),\infty)=C_k(h(1,\cdot),\infty)\ \mbox{for all}\ k\in\NN_0,\nonumber\\
			&\Rightarrow&C_k(\varphi,\infty)=C_k(\psi_{\lambda},\infty)\ \mbox{for all}\ k\in\NN_0.
		\end{eqnarray}
		
		Since $\lambda\not\in\sigma_0(p)$, we have
		\begin{eqnarray}\label{eq60}
			&&K_{\psi_{\lambda}}=\{0\}\nonumber,\\
			&\Rightarrow&C_k(\psi_{\lambda},\infty)=C_k(\psi_{\lambda},0)\ \mbox{for all}\ k\in\NN_0,\\
			&\Rightarrow&C_0(\psi_{\lambda},\infty)=C_1(\psi_{\lambda},\infty)=0\ (\mbox{see Proposition \ref{prop9}}),\nonumber\\
			&\Rightarrow&C_0(\varphi,\infty)=C_1(\varphi,\infty)=0\ (\mbox{see (\ref{eq59})}).\nonumber
		\end{eqnarray}
		
		Next we show that $C_m(\varphi,\infty)\neq 0$.
		
		We introduce the following two sets
		\begin{eqnarray*}
			&&G_r=\{u\in W^{1,p}(\Omega):\vartheta(u)<\lambda||u||^p_p,||u||=r\}\ (r>0),\\
			&&H=\{u\in W^{1,p}(\Omega):\vartheta(u)\geq \lambda||u||^p_p\}.
		\end{eqnarray*}
		
		There are symmetric sets and $G_r\cap H=\emptyset$. Also let
		$$\partial B_r=\{u\in W^{1,p}(\Omega):||u||=r\}.$$
		
		This is a Banach $C^1$-manifold and so locally contractible. The set $G_r\subseteq\partial B_r$ is open and so locally contractible too. The set $W^{1,p}(\Omega)\backslash H$ is open and of course locally contractible. By ${\rm ind}\,(\cdot)$ we denote the Fadell \& Rabinowitz \cite{7} cohomological index. Since $\lambda\in(\hat{\lambda}_m,\hat{\lambda}_{m+1})\backslash\sigma_0(p)$, we have
		$${\rm ind}\,G_r={\rm ind}\,(W^{1,p}(\Omega)\backslash H)=m.$$
		
		Theorem 3.6 of Cingolani \& Degiovanni \cite{3} implies that we can find $K\subseteq W^{1,p}(\Omega)$ compact such that $(G_r\cup K,G_r)$ homologically links $H$ in dimension $m\geq 2$ (see Motreanu, Motreanu \& Papageorgiou \cite[p. 167]{14}). So, Theorem 3.2 of \cite{3} says that
		\begin{eqnarray*}
			&&C_m(\psi_{\lambda},0)\neq 0,\\
			&\Rightarrow&C_m(\psi_{\lambda},\infty)\neq 0\ (\mbox{see (\ref{eq60})}),\\
			&\Rightarrow&C_m(\varphi,\infty)\neq 0\ (\mbox{see (\ref{eq59})}).
		\end{eqnarray*}
\end{proof}

Now we are ready to prove our first existence theorem.
\begin{theorem}\label{th14}
	If hypotheses $H(\xi),H(\beta),H_1$ hold, then problem (\ref{eq1}) admits a nontrivial solution $\hat{u}\in C^1(\overline{\Omega})$.
\end{theorem}
\begin{proof}
	From Proposition \ref{prop11} and \ref{prop12} we have
	$$C_1(\varphi,0)\neq 0\ \mbox{and}\ C_1(\varphi,\infty)=0.$$
	
	So, Corollary 6.92 of Motreanu, Motreanu \& Papageorgiou \cite[p. 173]{14}, implies that we can find $u_0\in K_{\varphi}$ such that
	\begin{eqnarray*}
		&&\varphi(u_0)<\varphi(0)=0\ \mbox{and}\ C_0(\varphi,u_0)\neq 0\\
		\mbox{or}&&\varphi(u_0)>\varphi(0)=0\ \mbox{and}\ C_2(\varphi,u_0)\neq 0.
	\end{eqnarray*}
	
	Evidently in both cases $u_0\neq 0$ and solves problem (\ref{eq1}). Moreover, from Papageorgiou \& R\u adulescu \cite{19} we have $u_0\in L^{\infty}(\Omega)$ and then Theorem 2 of Lieberman \cite{13} implies that $u_0\in C^1(\overline{\Omega})$.
\end{proof}

We can obtain another existence theorem if we change the geometry of the problem near the origin. In this case we allow resonance at $+\infty$ with any variational eigenvalue.

So, the new hypotheses on the reaction $f(z,x)$ are the following:

\smallskip
$H_2:$ $f:\Omega\times\RR\rightarrow\RR$ is a Carath\'eodory function such that $f(z,0)=0$ for almost all $z\in\Omega$ and
\begin{itemize}
	\item[(i)] for every $\rho>0$, there exists $a_{\rho}\in L^{\infty}(\Omega)$, such that
	$$|f(z,x)|\leq a_{\rho}(z)\ \mbox{for almost all}\ z\in\Omega,\ \mbox{all}\ |x|\leq\rho;$$
	\item[(ii)] there exists $m\in\NN$ such that
	$$\lim\limits_{x\rightarrow\pm\infty}\frac{f(z,x)}{|x|^{p-2}x}=\hat{\lambda}_m\ \mbox{uniformly for almost all}\ z\in\Omega;$$
	\item[(iii)] if $F(z,x)=\int^x_0f(z,s)ds$, then $\lim\limits_{x\rightarrow\pm\infty}[f(z,x)x-pF(z,x)]=+\infty$ uniformly for almost all $z\in\Omega;$
	\item[(iv)] there exists $\eta\in L^{\infty}(\Omega),\ \eta(z)\leq\hat{\lambda}_1$ for almost all $z\in\Omega$, $\eta\not\equiv\hat{\lambda}_1$ such that
	$$\limsup\limits_{x\rightarrow 0}\frac{pF(z,x)}{|x|^p}\leq\eta(z)\ \mbox{uniformly for almost all}\ z\in\Omega.$$
\end{itemize}
\begin{lemma}\label{lem15}
	If $\eta\in L^{\infty}(\Omega),\ \eta(z)\leq\hat{\lambda}_1$ for almost all $z\in\Omega,\ \eta\not\equiv\hat{\lambda}_1$, then there exists $c_6>0$ such that
	$$\psi(u)=\vartheta(u)-\int_{\Omega}\eta(z)|u|^pdz\geq c_6||u||^p\ \mbox{for all}\ u\in W^{1,p}(\Omega).$$
\end{lemma}
\begin{proof}
	From (\ref{eq3}) we see that $\psi\geq 0$. Arguing indirectly, suppose that the lemma is not true. Exploiting the $p$-homogeneity of $\psi(\cdot)$ we can find $\{u_n\}_{n\geq 1}\subseteq W^{1,p}(\Omega)$ such that
	\begin{equation}\label{eq61}
		||u_n||=1\ \mbox{for all}\ n\in\NN\ \mbox{and}\ \psi(u_n)\downarrow 0.
	\end{equation}
	
	We may assume that
	$$u_n\stackrel{w}{\rightarrow}u\ \mbox{in}\ W^{1,p}(\Omega)\ \mbox{and}\ u_n\rightarrow u\ \mbox{in}\ L^p(\Omega)\ \mbox{and in}\  L^p(\partial\Omega).$$
	
	Then in the limit as $n\rightarrow\infty$, we have
	\begin{eqnarray}\label{eq62}
		&&\psi(u)\leq 0\ (\mbox{see (\ref{eq6})}),\nonumber\\
		&\Rightarrow&\vartheta(u)\leq\int_{\Omega}\eta(z)|u|^pdz\leq\hat{\lambda}_1||u||^p_p,\\
		&\Rightarrow&\vartheta(u)=\hat{\lambda}_1||u||^p_p\ (\mbox{see (\ref{eq3})}),\nonumber\\
		&\Rightarrow&u=k\hat{u}_1\ \mbox{with}\ k\in\RR.\nonumber
	\end{eqnarray}
	
	If $k=0$, then $u=0$ and we have $u_n\rightarrow 0$ in $W^{1,p}(\Omega)$, a contradiction (see (\ref{eq61})).
	
	If $k\neq 0$, then $|u(z)|\neq 0$ for all $z\in\overline{\Omega}$ (recall that $\hat{u}_1\in D_+$). From (\ref{eq62}) and the hypothesis on $\eta(\cdot)$, we infer that
	$$\vartheta(u)<\hat{\lambda}_1||u||^p_p,$$
	which contradicts (\ref{eq3}). Therefore the lemma is true.
\end{proof}

Now we can have our second existence theorem.
\begin{theorem}\label{th16}
	If hypotheses $H(\xi),H(\beta),H_2$ hold, then problem (\ref{eq1}) admits a nontrivial solution $u_0\in C^1(\overline{\Omega})$.
\end{theorem}
\begin{proof}
	Hypothesis $H_2(iv)$ implies that given $\epsilon>0$, we can find $\delta=\delta(\epsilon)>0$ such that
	\begin{equation}\label{eq63}
		F(z,x)\leq\frac{1}{p}(\eta(z)+\epsilon)|x|^p\ \mbox{for almost all}\ z\in\Omega,\ \mbox{all}\ |x|\leq\delta.
	\end{equation}
	
	Let $u\in C^1(\overline{\Omega})$ with $||u||_{C^1(\overline{\Omega})}\leq\delta$. We have
	\begin{eqnarray*}
		\varphi(u)&=&\frac{1}{p}\vartheta(u)-\int_{\Omega}F(z,u)dz\\
		&\geq&\frac{1}{p}\left[\vartheta(u)-\int_{\Omega}\eta(z)|u|^pdz\right]-\epsilon||u||^p\ (\mbox{see (\ref{eq63})})\\
		&\geq&(c_7-\epsilon)||u||^p\ \mbox{for some}\ c_7>0\ \mbox{(see Lemma \ref{lem15})}.
	\end{eqnarray*}
	
	Choosing $\epsilon\in(0,c_7)$ we infer that
	\begin{eqnarray}\label{eq64}
		&&\varphi(u)\geq 0\ \mbox{for all}\ u\in C^1(\overline{\Omega}),\ ||u||_{C^1(\overline{\Omega})}\leq\delta,\nonumber\\
		&\Rightarrow&u=0\ \mbox{is a local}\ C^1(\overline{\Omega})-\mbox{minimizer of}\ \varphi,\nonumber\\
		&\Rightarrow&u=0\ \mbox{is a local}\ W^{1,p}(\Omega)-\mbox{minimizer of}\ \varphi\ (\mbox{see Proposition \ref{prop6}}),\nonumber\\
		&\Rightarrow&C_k(\varphi,0)=\delta_{k,0}\ZZ\ \mbox{for all}\ k\in\NN_0.
	\end{eqnarray}
	
	On the other hand, from Proposition \ref{prop13} we have $C_m(\varphi,\infty)\neq 0$. So, from Theorem 6.62 of Motreanu, Motreanu \& Papageorgiou \cite[p. 160]{14}, we know that we can find $u_0\in K_{\varphi}$ such that
	\begin{equation}\label{eq65}
		C_m(\varphi,u_0)\neq 0,\ m\geq 1.
	\end{equation}
	
	Comparing (\ref{eq64}) and (\ref{eq65}), we infer that $u_0\neq 0$. As before the nonlinear regularity theory implies that $u_0\in C^1(\overline{\Omega})$.
\end{proof}

\section{Resonance with Respect to the Principal Eigenvalue}

In this section, we examine problems which are resonant with respect to the principal eigenvalue. The problem under consideration is the following:
\begin{eqnarray}\label{eq66}
	\left\{\begin{array}{ll}
		-\Delta_pu(z)+\xi(z)|u(z)|^{p-2}u(z)=\hat{\lambda}_1|u(z)|^{p-2}u(z)+g(z,u(z))&\mbox{in}\ \Omega,\\
		\frac{\partial u}{\partial n_p}+\beta(z)|u|^{p-2}u=0&\mbox{on}\ \partial\Omega.
	\end{array}\right\}
\end{eqnarray}

The hypotheses on the perturbation $g(z,x)$ are the following:

\smallskip
$H_3:$ $g:\Omega\times\RR\rightarrow\RR$ is a Carath\'eodory function such that $g(z,0)=0$ for almost all $z\in\Omega$ and
\begin{itemize}
	\item[(i)] if $G(z,x)=\int^x_0g(z,s)ds$, then there exist functions $G_{\pm}\in L^1(\Omega)$ such that
	\begin{eqnarray*}
		&&\int_{\Omega}G_{\pm}(z)dz\leq 0\\
		&&g(z,x)\rightarrow 0\ \mbox{and}\ G(z,x)\rightarrow G_{\pm}(z)\ \mbox{uniformly for almost all}\ z\in\Omega,\ \mbox{as}\ x\rightarrow\pm\infty;
	\end{eqnarray*}
	and for every $\rho>0$ there exists $a_{\rho}\in L^{\infty}(\Omega)_+$ such that
	$$|g(z,x)|\leq a_{\rho}(z)\ \mbox{for almost all}\ z\in\Omega,\ \mbox{all}\ |x|\leq\rho.$$
	\item[(ii)] $G(z,x)\leq\frac{1}{p}[\hat{\lambda}(p)-\hat{\lambda}_1]|x|^p$ for almost all $z\in\Omega$, all $x\in\RR$;
	\item[(iii)] there exists a function $\eta\in L^{\infty}(\Omega)$ such that
	\begin{eqnarray*}
		&&\eta(z)\geq 0\ \mbox{for almost all}\ z\in\Omega,\ \eta\neq 0,\\
		&&\liminf\limits_{x\rightarrow 0}\frac{pG(z,x)}{|x|^p}\geq\eta(z)\ \mbox{uniformly for almost all}\ z\in\Omega.
	\end{eqnarray*}
\end{itemize}
\begin{remark}
	Because of hypothesis $H_3(i)$ in the terminology introduced by Landesman \& Lazer \cite{11}, the problem is ``strongly resonant" with respect to the principal eigenvalue. Such problems exhibit a partial lack of compactness (that is, the energy (Euler) functional of the problem, does not satisfy the C-condition at all levels). This is evident in Proposition \ref{prop17} which follows.
\end{remark}

The energy functional $\varphi:W^{1,p}(\Omega)\rightarrow\RR$ is defined by
$$\varphi(u)=\frac{1}{p}\vartheta(u)-\frac{\hat{\lambda}_1}{p}||u||^p_p-\int_{\Omega}G(z,u)dz\ \mbox{for all}\ u\in W^{1,p}(\Omega).$$

We have $\varphi\in C^1(W^{1,p}(\Omega)).$
\begin{prop}\label{prop17}
	If hypotheses $H(\xi),H(\beta),H_3$ hold, then the functional $\varphi$ satisfies the $C_c$-condition for every
	$$c<\min\left\{-\int_{\Omega}G_+(z)dz,-\int_{\Omega}G_-(z)dz\right\}.$$
\end{prop}
\begin{proof}
	Let $m_0=\min\left\{-\int_{\Omega}G_+(z)dz,-\int_{\Omega}G_-(z)dz\right\}$ and let $c<m_0$. We consider a sequence such that
	\begin{eqnarray}
		&&\varphi(u_n)\rightarrow c,\label{eq67}\\
		&&(1+||u_n||)\varphi'(u_n)\rightarrow 0\ \mbox{in}\ W^{1,p}(\Omega)^*.\label{eq68}
	\end{eqnarray}
	\begin{claim}
		$\{u_n\}_{n\geq 1}\subseteq W^{1,p}(\Omega)$ is bounded.
	\end{claim}
	
	Arguing by contradiction, suppose that the claim is not true. By passing to a subsequence if necessary, we may assume that
	\begin{equation}\label{eq69}
		||u_n||\rightarrow\infty .
	\end{equation}
	
	Let $y_n=\frac{u_n}{||u_n||},\ n\in\NN$. Then $||y_n||=1$ for all $n\in\NN$ and so we may assume that
	\begin{equation}\label{eq70}
		y_n\stackrel{w}{\rightarrow}y\ \mbox{in}\ W^{1,p}(\Omega)\ \mbox{and}\ y_n\rightarrow y\ \mbox{in}\ L^p(\Omega)\ \mbox{and in}\ L^p(\partial\Omega).
	\end{equation}
	
	From (\ref{eq67}) we have
	\begin{eqnarray*}
		&&\frac{1}{p}\vartheta(u_n)-\frac{\hat{\lambda}_1}{p}||u_n||^p_p-\int_{\Omega}G(z,u_n)dz\leq M_3\ \mbox{for some}\ M_3>0,\ \mbox{all}\ n\in\NN,\\
		&\Rightarrow&\frac{1}{p}\vartheta(y_n)-\frac{\hat{\lambda}_1}{p}||y_n||^p_p-\int_{\Omega}\frac{G(z,u_n)}{||u_n||^p}dz\leq\frac{M_3}{||u_n||^p}\ \mbox{for all}\ n\in\NN,\\
		&\Rightarrow&\vartheta(y)\leq\hat{\lambda}_1||y||^p_p\ (\mbox{see (\ref{eq69}), (\ref{eq70}) and hypothesis $H_3(i)$}),\\
		&\Rightarrow&\vartheta(y)=\hat{\lambda}_1||y||^p_p\ (\mbox{see (\ref{eq3})}),\\
		&\Rightarrow&y=k\hat{u}_1\ \mbox{with}\ k\in\RR.
	\end{eqnarray*}
	
	If $k=0$, then $y_n\rightarrow 0$ in $W^{1,p}(\Omega)$ a contradiction to the fact that $||y_n||=1$ for all $n\in\NN$.
	
	If $k\neq 0$, then to fix things we assume that $k>0$ (the reasoning is similar if $k<0$). We have
	\begin{equation}\label{eq71}
		u_n(z)\rightarrow+\infty\ \mbox{for almost all}\ z\in\Omega,\ \mbox{as}\ n\rightarrow\infty.
	\end{equation}
	
	From (\ref{eq67}) we see that given $\epsilon>0$, we can find $n_0=n_0(\epsilon)\in\NN$ such that
	\begin{eqnarray}\label{eq72}
		&&\varphi(u_n)\leq c+\epsilon\ \mbox{for all}\ n\geq n_0,\nonumber\\
		&\Rightarrow&\frac{1}{p}\vartheta(u_n)-\frac{\hat{\lambda}_1}{p}||u_n||^p_p-\int_{\Omega}G(z,u_n)dz\leq c+\epsilon\ \mbox{for all}\ n\geq n_0.
	\end{eqnarray}
	
	From (\ref{eq3}) we have
	$$\hat{\lambda}_1||u_n||^p_p\leq\vartheta(u_n)\ \mbox{for all}\ n\in\NN .$$
	
	Using this in (\ref{eq72}), we obtain
	\begin{eqnarray*}
		&&-\int_{\Omega}G(z,u_n)dz\leq c+\epsilon\ \mbox{for all}\ n\geq n_0,\\
		&\Rightarrow&-\int_{\Omega}G_+(z)dz\leq c+\epsilon\ (\mbox{from (\ref{eq71}), hypothesis}\ H_3(i)\ \mbox{and Fatou's lemma}).
	\end{eqnarray*}
	
	Since $\epsilon>0$ is arbitrary, we let $\epsilon\rightarrow 0^+$ and conclude that
	$$-\int_{\Omega}G_+(z)dz\leq c,$$
	which contradicts the choice of $c<m_0$. This proves the claim.
	
	Because of the claim we may assume that
	\begin{equation}\label{eq73}
		u_n\stackrel{w}{\rightarrow}u\ \mbox{in}\ W^{1,p}(\Omega)\ \mbox{and}\ u_n\rightarrow u\ \mbox{in}\ L^p(\Omega)\ \mbox{and in}\ L^p(\partial\Omega).
	\end{equation}
	
	From (\ref{eq68}) we have
	\begin{eqnarray}\label{eq74}
			&&\left|\left\langle A(u_n),h\right\rangle+\int_{\Omega}\xi(z)|u_n|^{p-2}u_nhdz+\int_{\partial\Omega}\beta(z)|u_n|^{p-2}u_nhd\sigma-\hat{\lambda}_1\int_{\Omega}|u_n|^{p-2}u_nhdz-\right.\nonumber\\
			&&\left.\int_{\Omega}g(z,u_n)hdz\right|\leq\frac{\epsilon_n||h||}{1+||u_n||}
	\end{eqnarray}
	for all $h\in W^{1,p}(\Omega)$ with $\epsilon_n\rightarrow 0^+$.
	
	In (\ref{eq74}) we choose $h=u_n-u\in W^{1,p}(\Omega)$, pass to the limit as $n\rightarrow\infty$ and use (\ref{eq73}). Then
	\begin{eqnarray*}
		&&\lim\limits_{n\rightarrow\infty}\left\langle A(u_n),u_n-u\right\rangle=0,\\
		&\Rightarrow&u_n\rightarrow u\ \mbox{in}\ W^{1,p}(\Omega)\ (\mbox{see Proposition \ref{prop7}}),\\
		&\Rightarrow&\varphi\ \mbox{satisfies the}\ C_c-\mbox{condition for}\ c<m_0.
	\end{eqnarray*}
The proof is complete.
\end{proof}

Now we prove a multiplicity theorem for problem (\ref{eq66}) producing two nontrivial smooth solutions.
\begin{theorem}\label{th18}
	If hypotheses $H(\xi),H(\beta),H_3$ hold, then problem (\ref{eq66}) admits at least two nontrivial solutions
	$$u_0,\hat{u}\in C^1(\overline{\Omega}).$$
\end{theorem}
\begin{proof}
	Hypothesis $H_3(iii)$ implies that given $\epsilon>0$, we can find $\delta=\delta(\epsilon)>0$ such that
	\begin{equation}\label{eq75}
		G(z,x)\geq\frac{1}{p}(\eta(z)-\epsilon)|x|^p\ \mbox{for almost all}\ z\in\Omega,\ \mbox{all}\ |x|\leq\delta.
	\end{equation}
	
	Recall that $\hat{u}_1\in D_+$. So, for $t\in(0,1)$ small we will have
	\begin{equation}\label{eq76}
		t\hat{u}_1(z)\in\left(0,\delta\right]\ \mbox{for all}\ z\in\overline{\Omega}.
	\end{equation}
	
	Then we have
	\begin{eqnarray}\label{eq77}
		\varphi(t\hat{u}_1)&=&\frac{t^p}{p}[\vartheta(\hat{u}_1)-\hat{\lambda}_1]-\int_{\Omega}G(z,t\hat{u}_1)dz\ (\mbox{recall}\ ||\hat{u}_1||_p=1)\nonumber\\
		&=&-\int_{\Omega}G(z,t\hat{u}_1)dz\nonumber\\
		&\leq&\frac{t^p}{p}\left[\epsilon-\int_{\Omega}\eta(z)\hat{u}_1^pdz\right]\ (\mbox{see (\ref{eq75}), (\ref{eq76}) and recall}\ ||\hat{u}_1||_p=1).
	\end{eqnarray}
	
	Since $\hat{u}_1\in D_+$ and $\eta\not\equiv 0$ (see hypothesis $H_3(iii)$), we have
	$$0<\tau_0=\int_{\Omega}\eta(z)\hat{u}_1^pdz.$$
	
	Then from (\ref{eq77}) and by choosing $\epsilon\in(0,\tau_0)$ we obtain
	\begin{equation}\label{eq78}
		\varphi(t\hat{u}_1)<0.
	\end{equation}
	
	Because $\vartheta(u)\geq\hat{\lambda}_1||u||^p_p$ for all $u\in W^{1,p}(\Omega)$ (see (\ref{eq3})) and using hypothesis $H_3(i)$ we infer that
	$$\varphi\ \mbox{is bounded below}.$$
	
	So, in conjunction with (\ref{eq78}) we have
	\begin{equation}\label{eq79}
		-\infty<m=\inf\varphi<0=\varphi(0).
	\end{equation}
	
	From hypothesis $H_3(i)$ we see that $m_0\geq 0$. Therefore Proposition \ref{prop17} implies that
	$$\varphi\ \mbox{satisfies the}\ C_m-\mbox{condition}.$$
	
	Invoking Proposition \ref{prop1}, we can find $u_0\in W^{1,p}(\Omega)$ such that
	\begin{eqnarray*}
		&&m=\varphi(u_0)<0=\varphi(0)\ (\mbox{see (\ref{eq79})}),\\
		&\Rightarrow&u_0\neq 0.
	\end{eqnarray*}
	
	Since $u_0\in K_{\varphi}$, it follows that $u_0$ is a nontrivial solution of (\ref{eq66}) and as before the nonlinear regularity theory implies that
	$$u_0\in C^1(\overline{\Omega}).$$
	
	Next we consider the following direct sum decomposition
	$$W^{1,p}(\Omega)=\RR\hat{u}_1\oplus V_p.$$
	
	If $u\in V_p$, then
	\begin{eqnarray}\label{eq80}
		\varphi(u)&\geq&\frac{1}{p}\vartheta(u)-\frac{\hat{\lambda}_1}{p}||u||^p_p-\frac{1}{p}[\hat{\lambda}(p)-\hat{\lambda_1}]||u||^p_p\nonumber\\
		&&(\mbox{see hypothesis}\ H_3(ii))\nonumber\\
		&=&\frac{1}{p}\vartheta(u)-\frac{\hat{\lambda}(p)}{p}||u||^p_p\nonumber\\
		&\geq&0\ \mbox{(see (\ref{eq6}))}\nonumber\\
		\Rightarrow&&\inf\limits_{V_p}\varphi\geq 0.
	\end{eqnarray}
	
	On the other hand, if $r\in(0,1)$ is small, then from (\ref{eq78}) and (\ref{eq76}) we see that
	\begin{equation}\label{eq81}
		\mu=\sup[\varphi(u):u\in\bar{B}_r\cap\RR\hat{u}_1]<0
	\end{equation}
	with $\bar{B}_r=\{u\in W^{1,p}(\Omega):||u||\leq r\}$. We consider the following family of maps
	\begin{equation}\label{eq82}
	\Gamma=\left\{\gamma\in C(\bar{B}_r\cap\RR\hat{u}_1,W^{1,p}(\Omega)):\gamma|_{\partial\bar{B}_r\cap\RR\hat{u}_1}=
id|_{\partial\bar{B}_r\cap\RR\hat{u}_1}\right\}.
	\end{equation}
	
	We assume that
	\begin{equation}\label{eq83}
		K_{\varphi}=\{0,u_0\}.
	\end{equation}
	
	Otherwise we already have a second nontrivial solution $\hat{u}_1$, which by the nonlinear regularity theory belongs in $C^1(\overline{\Omega})$ and so we are done.
	
	Let $b=0$ and $a=m=\varphi(u_0)$ and let $h(t,u)$ be the deformation postulated by Theorem \ref{th2} (the second deformation theorem). From (\ref{eq83}) we see that $K^b_{\varphi}=\{0\}$ and $\varphi^a=\{u_0\}$. Hence
	\begin{equation}\label{eq84}
		h(1,u)=u_0\ \mbox{for all}\ u\in V_p.
	\end{equation}
	
	Also, if $||u||=\frac{r}{2}$, then
	$$h\left(\frac{2(r-||u||)}{r},\frac{ru}{||u||}\right)=h(1,2u)=u_0\ (\mbox{since}\ 2||u||=r,\ \mbox{see (\ref{eq81}), (\ref{eq84})}).$$
	
	So, if we consider the map $\gamma_0:\bar{B}_r\cap\RR\hat{u}_1\rightarrow W^{1,p}(\Omega)$ defined by
	$$\gamma_0(u)=\left\{\begin{array}{ll}
		u_0&\mbox{if}\ ||u||<\frac{r}{2}\\
		h\left(\frac{2(r-||u||)}{r},\frac{ru}{||u||}\right)&\mbox{if}\ ||u||\geq\frac{r}{2},
	\end{array}\right.$$
	then from the previous remarks we see that $\gamma_0$ is continuous. Also, if $||u||=r$ then
	\begin{eqnarray*}
		&&h(0,u)=u,\\
		&\Rightarrow&\gamma_0|_{\partial\bar{B}_r\cap\RR\hat{u}_1}=id|_{\partial\bar{B}_r\cap\RR\hat{u}_1},\\
		&\Rightarrow&\gamma_0\in\Gamma\ (\mbox{see (\ref{eq82})}).
	\end{eqnarray*}
	
	From Theorem \ref{th2} (the second deformation theorem), we know that the homotopy $h(t,u)$ is $\varphi$-decreasing. Thus, from (\ref{eq81}) it follows that
	\begin{equation}\label{eq85}
		\varphi(\gamma_0(u))<0\ \mbox{for all}\ u\in\bar{B}_r\cap\RR u_1.
	\end{equation}
	
	From Example 5.2.3(b) of Gasinski \& Papageorgiou \cite[p. 642]{8}, we know that the sets $\partial\bar{B}_r\cap\RR \hat{u}_1$ and $V_p$ link in $W^{1,p}(\Omega)$ (see Definition \ref{def4}). Therefore we have
	\begin{eqnarray*}
		&&\gamma(\bar{B}_r\cap\RR \hat{u}_1)\cap V_p\neq\emptyset\ \mbox{for all}\ \gamma\in\Gamma,\\
		&\Rightarrow&\sup[\gamma(u):u\in \bar{B}_r\cap\RR \hat{u}_1]\geq 0\ \mbox{for all}\ \gamma\in\Gamma\ (\mbox{see (\ref{eq80})}),\\
		&\Rightarrow&\sup[\gamma_0(u):u\in \bar{B}_r\cap\RR \hat{u}_1]=\gamma_0(\tilde{u})\geq 0\ \mbox{for some}\ \tilde{u}\in \bar{B}_r\cap\RR \hat{u}_1.
	\end{eqnarray*}
	
	This contradicts (\ref{eq85}). So, we can find $\hat{u}\in K_{\varphi},\ \hat{u}\not\in\{0,u_0\}$. Then $\hat{u}$ is the second nontrivial solution of (\ref{eq66}). As before, the nonlinear regularity theory implies that $\hat{u}\in C^1(\overline{\Omega})$.
\end{proof}

We can have a three solutions theorem, if we change the geometry.

So, the new conditions on the perturbation term $g(z,x)$ are the following:

\smallskip
$H_4:$ $g:\Omega\times\RR\rightarrow\RR$ is a Carath\'eodory function such that $g(z,0)=0$ for almost all $z\in\Omega$ and
\begin{itemize}
	\item[(i)] if $G(z,x)=\int^x_0g(z,s)ds$, then there exist functions $G_{\pm}\in L^1(\Omega)$ and constants $c_-<0<c_+$ such that
	\begin{eqnarray*}
		&&0<\int_{\Omega}G_{\pm}(z)dz\leq\int_{\Omega}G(z,c_{\pm}\hat{u}_1)dz\\
		&&pG(z,x)-g(z,x)x\rightarrow G_{\pm}(z)\ \mbox{uniformly for almost all}\ z\in\Omega\ \mbox{as}\ x\rightarrow\pm\infty
	\end{eqnarray*}
	and there exist $a\in L^{\infty}(\Omega)_+$ and $1<r<p^*=\left\{\begin{array}{ll}
		\frac{Np}{N-p}&\mbox{if}\ p<N\\
		+\infty&\mbox{if}\ p\geq N
	\end{array}\right.$ such that
	$$|g(z,x)|\leq a(z)(1+|x|^{r-1})\ \mbox{for almost all}\ z\in\Omega,\ \mbox{all}\ x\in\RR;$$
	\item[(ii)] $G(z,x)\leq\frac{1}{p}[\hat{\lambda}(p)-\hat{\lambda}_1]|x|^p$ for almost all $z\in\Omega$, all $x\in\RR$;
	\item[(iii)] there exists a function $\eta\in L^{\infty}(\Omega)_+$ such that
	\begin{eqnarray*}
		&&\eta(z)\leq 0\ \mbox{for almost all}\ z\in\Omega,\ \eta\neq 0,\\
		&&\limsup\limits_{x\rightarrow 0}\frac{pG(z,x)}{|x|^p}\leq\eta(z)\ \mbox{uniformly for almost all}\ z\in\Omega.
	\end{eqnarray*}
\end{itemize}
\begin{remark}
	Again hypothesis $H_4(i)$ incorporates in our framework problems which are strongly resonant with respect to the principal eigenvalue.
\end{remark}
\begin{prop}\label{prop19}
	If hypotheses $H(\xi),H(\beta),H_4$ hold, then the functional $\varphi$ satisfies the $C_c$-condition for all $c\neq-\frac{1}{p}\int_{\Omega}G_{\pm}(z)dz$.
\end{prop}
\begin{proof}
	Consider a Cerami sequence $\{u_n\}_{n\geq 1}\subseteq W^{1,p}(\Omega)$ (that is, the sequence satisfies (\ref{eq67}) and (\ref{eq68})). We show that $\{u_n\}_{n\geq 1}\subseteq W^{1,p}(\Omega)$ is bounded. Arguing by contradiction, assume that $||u_n||\rightarrow\infty$ and let $y_n=\frac{u_n}{||u_n||},\ n\in\NN$. Since $||y_n||=1$ for all $n\in\NN$, we may assume that
	$$y_n\stackrel{w}{\rightarrow}y\ \mbox{in}\ W^{1,p}(\Omega)\ \mbox{and}\ y_n\rightarrow y\ \mbox{in}\ L^p(\Omega)\ \mbox{and in}\ L^p(\partial\Omega).$$
	
	Reasoning as in the proof of Proposition \ref{prop17}, we show that $y=k\hat{u}_1$, with $k\neq 0$. To fix things we assume that $k>0$ and so $u_n(z)\rightarrow+\infty$ for almost all $z\in\Omega$.
	
	From (\ref{eq67}) and (\ref{eq68}) we have
	\begin{eqnarray}\label{eq86}
		&&p(c-\epsilon)\leq p\varphi(u_n)\leq p(c+\epsilon)\ \mbox{for all}\ n\geq n_0,\\
		&&|\left\langle \varphi'(u_n),h\right\rangle|\leq\frac{\epsilon_n||h||}{1+||u_n||}\ \mbox{for all}\ h\in W^{1,p}(\Omega)\ \mbox{with}\ \epsilon_n\rightarrow 0^+.\nonumber
	\end{eqnarray}
	
	Choosing $h=u_n\in W^{1,p}(\Omega)$ we obtain
	\begin{equation}\label{eq87}
		-\epsilon_n\leq-\left\langle \varphi'(u_n),u_n\right\rangle\leq\epsilon_n\ \mbox{for all}\ n\in\NN.
	\end{equation}
	
	Adding (\ref{eq86}) and (\ref{eq87}), we obtain
	$$p(c-\epsilon)-\epsilon_n\leq\int_{\Omega}[g(z,u_n)u_n-pG(z,u_n)]dz\leq p(c+\epsilon)+\epsilon_n\ \mbox{for all}\ n\geq n_0.$$
	
	Recalling that $u_n(z)\rightarrow+\infty$ for almost all $z\in\Omega$ and that $\epsilon>0$ is arbitrary, if we pass to the limit as $n\rightarrow\infty$ and use hypothesis $H_4(i)$, we obtain
	$$c=-\frac{1}{p}\int_{\Omega}G_+(z)dz,$$
	a contradiction to our assumption on the level $c$.
\end{proof}

We consider the direct sum decomposition
$$W^{1,p}(\Omega)=\RR\hat{u}_1\oplus V_p$$
and we introduce the following two open subsets of $W^{1,p}(\Omega)$
$$U_+=\{t\hat{u}_1+v:t>0,v\in V_p\}\ \mbox{and}\ U_-=\{t\hat{u}_1+v:t<0,v\in V_p\}.$$

Note that
\begin{equation}\label{eq88}
	\inf\limits_{\bar{U}_{\pm}}\varphi\leq\varphi(c_{\pm}\hat{u}_1)=
-\int_{\Omega}G(z,c_{\pm}\hat{u}_1)dz\leq-\int_{\Omega}G_{\pm}(z)dz<0.
\end{equation}

Moreover, as in the proof of Theorem \ref{th18}, using hypothesis $H_4(iii)$, we have that
\begin{equation}\label{eq89}
	\inf\limits_{V_p}\varphi=0.
\end{equation}

These observations will help us to prove the existence of three nontrivial smooth solutions.
\begin{theorem}\label{th20}
	If hypotheses $H(\xi),H(\beta),H_4$ hold, then problem (\ref{eq66}) has at least three nontrivial solutions
\end{theorem}
$$\hat{u}_+,\ \hat{u}_-,\ \hat{y}\in C^1(\overline{\Omega}).$$
\begin{proof}
	Let $\varphi_+:W^{1,p}(\Omega)\rightarrow\overline{\RR}=\RR\cup\{+\infty\}$ be the lower semicontinuous and bounded below functional defined by
	\begin{equation*}
		\varphi_+(u)=\left\{\begin{array}{ll}
			\varphi(u) 	& \mbox{if}\ u\in\overline{U}_+ \\
			+\infty		& \mbox{otherwise}
		\end{array}\right.\ \mbox{(see hypothesis}\ H_4(i))
	\end{equation*}
	
	Invoking the extended Ekeland variational principle (see, for example, Gasinski \& Papageorgiou \cite[p. 598]{8}), we can find $\{u_n\}_{n\geq 1}\subseteq U_+$ such that
	\begin{eqnarray}
		\varphi(u_n)=\varphi_+(u_n)\downarrow\ \inf\varphi_+\ \mbox{(recall}\ \varphi_+\ \mbox{is bounded below)}\label{eq90},\\
		\varphi(u_n)=\varphi_+(u_n)\leq\varphi_+(y)+\frac{1}{n(1+||u_n||)}||y-u_n||\ \mbox{for all}\ y\in W^{1,p}(\Omega).\label{eq91}
	\end{eqnarray}
	
	Since $u_n\in U_+$ for $h\in W^{1,p}(\Omega)$ and $\lambda\in(0,1)$ is small we have
	$$u_n+\lambda h\in U_+.$$
	
	Because $\varphi_+|_{U_+}=\varphi|_{U_+}$, from (\ref{eq91}) with $y=u_n+\lambda h$, we have
	\begin{eqnarray*}
		&-&\frac{||h||}{n(1+||u_n||)}\leq\frac{\varphi(u_n+\lambda h)-\varphi(u_n)}{\lambda},\\
		\Rightarrow & - & \frac{||h||}{n(1+||u_n||)}\leq\left\langle \varphi'(u_n),h\right\rangle\,.
	\end{eqnarray*}
	
	From Lemma 5.1.38 of Gasinski \& Papageorgiou \cite[p. 639]{8}, we know that we can find $u^*_n\in W^{1,p}(\Omega)^*$ with $||u^*_n||_*\leq 1$ such that
	\begin{eqnarray}
		&	&\left\langle u^*_n,h\right\rangle\leq n(1+||u_n||)\langle\varphi'(u_n),h\rangle\ \mbox{for all}\ h\in W^{1,p}(\Omega) \nonumber \\
		& \Rightarrow & (1+||u_n||)\varphi'(u_n)\rightarrow 0\ \in W^{1,p}(\Omega)^* \nonumber \\
		& \Rightarrow & u_n\rightarrow \hat{u}_+\in W^{1,p}(\Omega)\ \mbox{(see (\ref{eq86}) and Theorem \ref{th18}),} \nonumber \\
		& \Rightarrow & \varphi(\hat{u}_+)=\inf\limits_{\overline{U}_+}\varphi=\inf\varphi_+<0\ \mbox{(see (\ref{eq88}))}. \label{eq92}
	\end{eqnarray}
	
	If $\hat{u}_+\in\partial U_+$, then $\hat{u}_+\in V_p$ and so
	$$\varphi(\hat{u}_+)\geq0\ \mbox{(see (\ref{eq89})).}$$
	
	This contradicts (\ref{eq92}). Therefore $\hat{u}_+\in U_+$ and this means that $\hat{u}_+$ is a local minimizer of $\varphi$. Moreover, the nonlinear regularity theory implies that $\hat{u}_+\in C^1(\overline{\Omega})$.
	
	Similarly using the lower semicontinuous and bounded below functional
	\begin{eqnarray*}
		\varphi_-(u)=\left\{\begin{array}{ll}
			\varphi(u) 	& \mbox{if}\ u\in\overline{U}_- \\
			+\infty		& \mbox{otherwise},
		\end{array}\right.
	\end{eqnarray*}
	we produce $\hat{u}_-\in C^1(\overline{\Omega})$ a second nontrivial smooth solution of (\ref{eq66}), which is also a local minimizer of $\varphi$.
	
	Without any loss of generality we may assume that
	$$\varphi(\hat{u}_-)\leq\varphi(\hat{u}_+)$$
	(the reasoning is similar if the opposite inequality holds). We assume that $K_\varphi$ is finite (otherwise we already have an infinity of nontrivial solutions of (\ref{eq66})) all of them in $C^1(\overline{\Omega})$ by the nonlinear regularity theory). Since $\hat{u}_+$ is a local minimizer of $\varphi$, we can find $\rho\in(0,1)$ small such that
	\begin{eqnarray}\label{eq93}
		\varphi(\hat{u}_-)\leq\varphi(\hat{u}_+)<\inf\left[\varphi(u):||u-\hat{u}_+||=\rho\right]=m^+_\rho\ ||\hat{u}_--\hat{u}_+||>\rho\\\
		\mbox{(see \cite{1}, proof of Proposition 29)}. \nonumber
	\end{eqnarray}
	
	Let $\Gamma=\{\gamma\in C([0,1],W^{1,p}(\Omega)):\gamma(0)=\hat{u}_-,\gamma(1)=\hat{u}_+\}$ and define
	$$c=\inf\limits_{\gamma\in\Gamma}\max\limits_{0\leq t\leq1}\varphi(\gamma(t)).$$
	
	Since $\hat{u}_+\in U_+$ and $\hat{u}_-\in U_-$, from (\ref{eq89}) we see that
	\begin{eqnarray}\label{eq94}
		& &c\geq0 \nonumber \\
		& \Rightarrow & \varphi\ \mbox{satisfies the}\ C_c-\mbox{condition} \\
		& &\mbox{(see Theorem \ref{th18} and hypothesis }H_1(i)). \nonumber
	\end{eqnarray}
	Then (\ref{eq93}) and (\ref{eq94}) permit the use of Theorem \ref{th5} (the mountain pass theorem). So, we can find $\hat{y}\in W^{1,p}(\Omega)$ such that
	\begin{eqnarray*}
		&	& \hat{y}\in K_\varphi\ \mbox{and}\ m^+_\rho\leq\varphi(\hat{y}) \\
		& \Rightarrow & \hat{y}\in C^1(\overline{\Omega})\ \mbox{is a solution of (\ref{eq66}) and}\ \hat{y}\notin\{\hat{u}_+,\hat{u}_-\}\ \mbox{(see (\ref{eq93}))}.
	\end{eqnarray*}
	
	Since $\hat{y}$ is a critical point of $\varphi$ of mountain pass type, we have
	\begin{equation}\label{eq95}
		C_1(\varphi,\hat{y})\neq0
	\end{equation}
	(see Motreanu, Motreanu \& Papageorgiou \cite[p. 168]{14}).
	
	On the other hand, from hypothesis $H_4$(iii) we see that given $\epsilon>0$, we can find $\delta=\delta(\epsilon)>0$ such that
	\begin{equation}\label{eq96}
		G(z,x)\leq\frac{1}{p}(\eta(z)+\epsilon)|x|^p\ \mbox{for almost all}\ z\in\Omega,\ \mbox{all}\ |x|\leq\delta.
	\end{equation}
	Then for $u\in C^1(\overline{\Omega})$ with $||u||_{C^1(\overline{\Omega})}\leq\delta$, we have
	\begin{eqnarray}
		\varphi(u) 	& = 		& \frac{1}{p}\vartheta(u)-\frac{\hat{\lambda}_1}{p}\,||u||^p_p-\int_{\Omega}G(z,u)dz \nonumber \\
					& \geq 	& \frac{1}{p}\left[ \vartheta(u)-\int_\Omega [\hat{\lambda}_1+\eta(z)] |u|^pdz-\epsilon||u||^p\right]\ \mbox{(see (\ref{eq96}))} \nonumber \\
					& \geq	& (c_8-\epsilon)||u||^p\ \mbox{for some}\ c_8>0\ \mbox{(see Lemma \ref{lem15})}. \label{eq97}
	\end{eqnarray}
	
	Choosing $\epsilon\in(0,c_8)$, from (\ref{eq97}) we see that
	\begin{eqnarray}
		&	& u=0\ \mbox{is a local}\ C^1(\overline{\Omega})-\mbox{minimizer of}\ \varphi, \nonumber \\
		& \Rightarrow & u=0\ \mbox{is a local}\ W^{1,p}(\Omega)-\mbox{minimizer of}\ \varphi\ \mbox{(see Proposition \ref{prop6})} \nonumber \\
		& \Rightarrow & C_k(\varphi,0)=\delta_{k,0}\ZZ\ \mbox{for all}\ k\in\NN_0. \label{eq98}
	\end{eqnarray}
	
	Comparing (\ref{eq95}) and (\ref{eq98}), we infer that
	\begin{eqnarray*}
		&	& \hat{y}\neq0, \\
		& \Rightarrow & \hat{y}\in C^1(\overline{\Omega})\ \mbox{is the third nontrivial smooth solution of (\ref{eq66})}
	\end{eqnarray*}
 \end{proof}

\begin{remark}\label{rem20}
It is interesting to know if, at least in the semilinear case $p=2$, we can improve the above theorem and produce a fourth nontrivial solution. The failure of the compactness condition at certain levels (see Proposition \ref{prop19}) does not allow us to compute the critical groups at infinity (see Bartsch \& Li \cite{bart}) and therefore we cannot use the Morse relation (see \eqref{eq5}). So, a different approach is needed.
\end{remark}

 \medskip
{\bf Acknowledgements.} The authors wish to thank three anonymous referees for their corrections and remarks, which improved the paper.
This work was partially supported by the Slovenian Research Agency grants P1-0292, J1-8131, J1-7025 and J1-6721.
V.D.~R\u{a}dulescu was supported by a grant of the Romanian National Authority for Scientific Research and Innovation, CNCS-UEFISCDI, project number PN-III-P4-ID-PCE-2016-0130.

\end{document}